\documentclass[12pt]{amsart}
\usepackage{times,a4wide,mathrsfs,amssymb,amsmath,amsthm}

\newcommand{\C}{\mathbb{C}}

\newcommand{\QQ}{\mathbb{Q}}
\newcommand{\NN}{\mathbb{N}}
\newcommand{\PP}{\mathbb{P}}

\newcommand{\Sy}{\mathfrak S}
\newcommand{\Bb}{\mathfrak B}
\newcommand{\oo}{\mathfrak o}

\newcommand{\MM}{\mathcal M}

\newcommand{\gr}{\hbox{Gr}}
\newcommand{\wt}{\widetilde}
\newcommand{\rom}{\romannumeral}

\DeclareMathOperator{\aut}{Aut}
\DeclareMathOperator{\ide}{id}

\DeclareMathOperator{\ima}{Im}
\DeclareMathOperator{\ord}{ord}
\DeclareMathOperator{\rank}{rank}

\newtheorem{theorem}{Theorem}[section]

\newtheorem{lemma}[theorem]{Lemma}

\newtheorem{corollary}[theorem]{Corollary}
\newtheorem{proposition}[theorem]{Proposition}
\newtheorem{conjecture}[theorem]{Conjecture}
\newtheorem{remark}[theorem]{Remark}
\newtheorem{definition}[theorem]{Definition}
\newtheorem{convention}{Conventions}

\newtheorem{nonumbering}{Theorem}

\newtheorem{nonumberingc}{Corollary}

\newtheorem{nonumberingt}{Acknowledgements}

\begin{document}
\author[Robert Laterveer]
{Robert Laterveer}

\address{Institut de Recherche Math\'ematique Avanc\'ee,
CNRS -- Universit\'e 
de Strasbourg,\
7 Rue Ren\'e Des\-car\-tes, 67084 Strasbourg CEDEX,
FRANCE.}
\email{robert.laterveer@math.unistra.fr}

\title{ Algebraic cycles on some special hyperk\"ahler varieties}

\begin{abstract} This note contains some examples of hyperk\"ahler varieties $X$ having a group $G$ of non--symplectic automorphisms, and such that the action of $G$ on certain Chow groups of $X$ is as predicted by Bloch's conjecture. The examples range in dimension from $6$ to $132$. For each example, the quotient $Y=X/G$ is a Calabi--Yau variety which has interesting Chow--theoretic properties; in particular, the variety $Y$ satisfies (part of) a strong version of the Beauville--Voisin conjecture.
\end{abstract}


\maketitle

\section{Introduction}

Let $X$ be a hyperk\"ahler variety of dimension $n=2k$ (i.e., a projective irreducible holomorphic symplectic manifold, cf. \cite{Beau0}, \cite{Beau1}). Let $G\subset\aut(X)$ be a finite cyclic group of order $k$
consisting of non--symplectic automorphisms. We will be interested in the action of $G$ on the Chow groups $A^\ast(X)$.
 (Here, $A^i(X):=CH^i(X)_{\QQ}$ denotes the Chow group of codimension $i$ algebraic cycles modulo rational equivalence with $\QQ$--coefficients. We will write $A^i_{hom}(X)$ and $A^i_{AJ}(X)\subset A^i(X)$ to denote the subgroups of homologically trivial (resp. Abel--Jacobi trivial) cycles.)

We will suppose $X$ has a multiplicative Chow--K\"unneth decomposition, in the sense of \cite{SV}. This implies the Chow ring of $X$ is a bigraded ring $A^\ast_{(\ast)}(X)$, 
where each Chow group splits as
  \[ A^i(X) = \bigoplus_{j} A^i_{(j)}(X)\ ,\]
and the piece $A^i_{(j)}(X)$ is expected to be the graded $\gr^j_F A^i(X)$ for the conjectural Bloch--Beilinson filtration $F^\ast$ on Chow groups. (Conjecturally, all hyperk\"ahler varieties have a multiplicative Chow--K\"unneth decomposition. This has been checked for Hilbert schemes of $K3$ surfaces \cite{SV}, \cite{V6}, and for generalized Kummer varieties \cite{FTV}.)

Since $H^{n,0}(X)=H^{2,0}(X)^{\otimes k}$, the group $G$ acts as the identity on $H^{n,0}(X)$. For $i<n$, we have that $\sum_{g\in G} g^\ast$ acts as $0$ on $H^{i,0}(X)$. The Bloch--Beilinson conjectures \cite{J2} thus imply the following conjecture:

\begin{conjecture}\label{theconj} Let $X$ be a hyperk\"ahler variety of dimension $n=2k$, and let $G\subset\aut(X)$ be a finite cyclic group of order $k$ of non--symplectic automorphisms.
Then
  \[  \begin{split}   & A^n_{(n)}(X)\cap A^n(X)^G=A^n_{(n)}(X)\ ;\\
                           &A^n_{(j)}(X)\cap A^n(X)^G=0\ \ \ \hbox{for}\ 0<j<n\ ;\\
                           & A^i_{(i)}(X)\cap A^i(X)^G=0\ \ \ \hbox{for}\ 0<i<n\ .\\
                         \end{split}\]
   \end{conjecture}                        
(Here $A^i(X)^G\subset A^i(X)$ denotes the $G$--invariant part of the Chow group $A^i(X)$.)

The aim of this note is to find examples where conjecture \ref{theconj} is verified. The main result presents an example of dimension $n=6$ (and so $k=3$) where most of conjecture \ref{theconj} is true. The example is given by the Hilbert scheme of a certain special $K3$ surface studied by Livn\'e--Sch\"utt--Yui \cite{LSY}:
%$16$ examples (with $k$ ranging from $3$ to $66$) where conjecture \ref{theconj} holds true. These $16$ examples are constructed as follows: let $S_k$ be one of the $16$ special $K3$ surfaces studied by Livn\'e--Sch\"utt--Yui \cite{LSY} (these $K3$ surfaces will be called ``LSY surfaces'', cf. definition \ref{lsy}). Then, for the Hilbert schemes $X:=(S_k)^{[k]}$, a strong version of conjecture \ref{theconj} happens to be true:

\begin{nonumbering}[=theorem \ref{main3}] Let $S_3$ be the $K3$ surface as in theorem \ref{lsythm}, and let $X$ be the Hilbert scheme $X:=(S_3)^{[3]}$ of dimension $6$. Let $G\subset \aut(X)$ be the order $3$ group of non--symplectic natural automorphisms, corresponding to the group $G_{S_3}\subset\aut(S_3)$ of definition \ref{lsy}. Then
\[     A^i_{(j)}(X)\cap A^i(X)^G =
                                                                      %  A^i_{(j)}(X) &\hbox{if}\ (i,j)=(6,6)\ ;\\
                                                                            0 \ \ \ \hbox{if} \ (i,j)\in \  \Bigl\{ (2,2), (4,4), (3,2),(5,2),(6,2),(6,4)\Bigr\} .\\
                                                                      \]

\end{nonumbering}

The proof of theorem \ref{main3} is a fairly easy consequence of the fact that the surface $S_3$ (and hence the Hilbert scheme $X$) has {\em finite--dimensional motive\/} (in the sense of \cite{Kim}), and is $\rho$--maximal (in the sense of \cite{Beau2}). Yet, the implications of theorem \ref{main3} are quite striking. These implications are most conveniently presented in terms of the Chow ring of the quotient $Y=X/G$ (the variety $Y$ is a $6$--dimensional ``Calabi--Yau variety with quotient singularities''):

\begin{nonumberingc}[=corollary \ref{prod3}] Let $X$ and $G$ be as in theorem \ref{main3}, and let $Y:=X/G$. For any $r\in\NN$, let 
\[E^\ast(Y^r)\subset A^\ast(Y^r)\] 
  denote the subalgebra generated by (pullbacks of) $A^1(Y), A^2(Y), A^3(Y)$ and the diagonal $\Delta_Y\in A^6(Y\times Y)$ and the small diagonal $\Delta_Y^{sm}\in A^{12}(Y^3)$. Then the cycle class map
  \[ E^i(Y^r)\ \to\ H^{2i}(Y^r) \]
  is injective for $i\ge 6r-1$.
\end{nonumberingc}

\begin{nonumberingc}[=corollary \ref{prod1}] Let $X$ and $G$ be as in theorem \ref{main3}, and let $Y:=X/G$. Let $a\in A^i(Y)$ be a cycle with $i\not=3$. Assume $a$ is a sum of intersections of $2$ cycles of strictly positive codimension, i.e.
  \[ a\in \ima\Bigl(  A^m(Y)\otimes A^{i-m}(Y)\ \to\ A^i(Y)\Bigr)\ ,\ \ \ 0<m<i\ .\]
  Then $a$ is rationally trivial if and only if $a$ is homologically trivial.
\end{nonumberingc}

This behaviour is remarkable, because $A^6(Y)$ is ``huge'' (it is not supported on any proper subvariety). In a sense, corollary \ref{prod1} is a mixture of the Beauville--Voisin conjecture (concerning the Chow ring of Hilbert schemes of $K3$ surfaces \cite[Conjecture 1.3]{V12}) on the one hand, and results
concerning $0$--cycles on certain Calabi--Yau varieties \cite{V13}, \cite{LFu}, on the other hand (cf. remark \ref{compare}).
These corollaries are easily proven; one merely exploits the good properties of multiplicative Chow--K\"unneth decompositions combined with finite--dimensionality of the motive of $X$.

We also give a partial generalization of theorem \ref{main3} to Hilbert schemes of higher dimension. This generalization concerns Hilbert schemes of the other special $K3$ surfaces $S_k$ ($k>3$) studied by Livn\'e--Sch\"utt--Yui \cite{LSY}. The surfaces $S_k$ all have finite--dimensional motive, however (apart from $k=3$) they are {\em not\/} $\rho$--maximal; for this reason, the conclusion is weaker in these cases:

\begin{nonumbering}[=theorem \ref{main}] Let $S_k$ be one of the $16$ $K3$ surfaces studied in \cite{LSY}. Let $X$ be the Hilbert scheme $X=(S_k)^{[k]}$ of dimension $n=2k$. Let $G\subset\aut(X)$ be the order $k$ group of natural automorphisms induced by the order $k$ automorphisms of $S_k$. Then
  \[ A^i_{(2)}(X)\cap A^i(X)^G =0\ \ \ \hbox{for}\ i\in\{2,n\}\ .\]
\end{nonumbering}

The $K3$ surfaces $S_k$ of \cite{LSY} have $k$ ranging from $3$ to $66$; the dimension $n$ in theorem \ref{main} thus ranges from $6$ to $132$. 
Theorem \ref{main} as proven below is actually more general than the above statement: theorem \ref{main} also
applies to certain of the $K3$ surfaces studied in \cite{Schu} (in particular, there exists a one--dimensional family of Hilbert schemes $X$ of dimension $8$ for which theorem \ref{main} is true).

Again, the quotient $Y:=X/G$ is a ``Calabi--Yau variety with quotient singularities'' (of dimension $n$ up to $132$) which has interesting Chow--theoretic behaviour:

\begin{nonumberingc}[=corollaries \ref{cor1} and \ref{cor2}] Let $X$ and $G$ be as in theorem \ref{main}. Let $Y:=X/G$. 

\noindent
(\rom1) Let $a\in A^{n-1}(Y)$ be a $1$--cycle which is in the image of the intersection product map
  \[ A^{i_1}(Y)\otimes A^{i_2}(Y)\otimes\cdots\otimes A^{i_r}(Y)\ \to\ A^{n-1}(Y)\ ,\]
  where all $i_j$ are $\le 2$. Then $a$ is rationally trivial if and only if $a$ is homologically trivial.

\noindent
(\rom2) Let $a\in A^n(Y)$ be a $0$--cycle which is in the image of the intersection product map
  \[ A^3(Y)\otimes A^{i_1}(Y)\otimes\cdots\otimes A^{i_r}(Y)\ \to\ A^{n}(Y)\ ,\]
   where all $i_j$ are $\le 2$. Then $a$ is rationally trivial if and only if $a$ is homologically trivial.
  \end{nonumberingc} 

Results similar in spirit have been obtained for certain other hyperk\"ahler varieties and their Calabi--Yau quotients in \cite{EPW}, \cite{BlochHK4}.

 \vskip0.6cm

\begin{convention} In this article, the word {\sl variety\/} will refer to a reduced irreducible scheme of finite type over $\C$. A {\sl subvariety\/} is a (possibly reducible) reduced subscheme which is equidimensional. 

{\bf All Chow groups will be with rational coefficients}: we will denote by $A_j(X)$ the Chow group of $j$--dimensional cycles on $X$ with $\QQ$--coefficients; for $X$ smooth of dimension $n$ the notations $A_j(X)$ and $A^{n-j}(X)$ are used interchangeably. 

%We will write $B_j(X)$ for the group of $j$--dimensional cycles modulo algebraic equivalence with $\QQ$--coefficients; for $X$ smooth of dimension $n$ the notations $B_j(X)$ and $B^{n-j}(X)$ are used interchangeably. 

The notations $A^j_{hom}(X)$, $A^j_{AJ}(X)$ will be used to indicate the subgroups of homologically trivial, resp. Abel--Jacobi trivial cycles.
For a morphism $f\colon X\to Y$, we will write $\Gamma_f\in A_\ast(X\times Y)$ for the graph of $f$.
The contravariant category of Chow motives (i.e., pure motives with respect to rational equivalence as in \cite{Sc}, \cite{MNP}) will be denoted $\MM_{\rm rat}$.

%To avoid heavy notation, if $\tau\colon Y\to X$ is a closed inclusion and $a\in A_iY$, we will frequently write $a\in A_iX$ to indicate the proper push--forward $\tau_\ast(a)$. Likewise, for any inclusion $Y\subset X$ and $b\in A^jX$ we will often write
%  \[   b\vert_{Y}\ \ \in A^jY\]
 % to indicate the cycle class $\tau^\ast(b)$.

%The Griffiths group $\grif^j$ is the group of codimension $j$ cycles that are homologically trivial modulo algebraic equivalence, again with $\QQ$--coefficients. 

We will write $H^j(X)$ 
to indicate singular cohomology $H^j(X,\QQ)$.

Given a group $G\subset\aut(X)$ of automorphisms of $X$, we will write $A^j(X)^G$ (and $H^j(X)^G$) for the subgroup of $A^j(X)$ (resp. $H^j(X)$) invariant under $G$.
\end{convention}

\section{Preliminaries}

\subsection{Quotient varieties}

\begin{definition} A {\em projective quotient variety\/} is a variety
  \[ Y=X/G\ ,\]
  where $X$ is a smooth projective variety and $G\subset\aut(X)$ is a finite group.
  \end{definition}
  
 \begin{proposition}[Fulton \cite{F}]\label{quot} Let $Y$ be a projective quotient variety of dimension $n$. Let $A^\ast(Y)$ denote the operational Chow cohomology ring. The natural map
   \[ A^i(Y)\ \to\ A_{n-i}(Y) \]
   is an isomorphism for all $i$.
   \end{proposition}
   
   \begin{proof} This is \cite[Example 17.4.10]{F}.
      \end{proof}

\begin{remark} It follows from proposition \ref{quot} that the formalism of correspondences goes through unchanged for projective quotient varieties (this is also noted in \cite[Example 16.1.13]{F}). We can thus consider motives $(Y,p,0)\in\MM_{\rm rat}$, where $Y$ is a projective quotient variety and $p\in A^n(Y\times Y)$ is a projector. For a projective quotient variety $Y=X/G$, one readily proves (using Manin's identity principle) that there is an isomorphism
  \[  h(Y)\cong h(X)^G:=(X,\Delta^G_X,0)\ \ \ \hbox{in}\ \MM_{\rm rat}\ ,\]
  where $\Delta^G_X$ denotes the idempotent 
 \[ \Delta^G_X:=  {1\over \vert G\vert}{\sum_{g\in G}}\ \Gamma_g\ \ \ \in A^n(X\times X).\] 
 (NB: $\Delta^G_X$ is a projector on the $G$--invariant part of the Chow groups $A^\ast(X)^G$.) 
  \end{remark}

\subsection{Finite--dimensional motives}

We refer to \cite{Kim}, \cite{An}, \cite{MNP}, \cite{Iv}, \cite{J4} for basics on the notion of finite--dimensional motive. 
An essential property of varieties with finite--dimensional motive is embodied by the nilpotence theorem:

\begin{theorem}[Kimura \cite{Kim}]\label{nilp} Let $X$ be a smooth projective variety of dimension $n$ with finite--dimensional motive. Let $\Gamma\in A^n(X\times X)_{}$ be a correspondence which is numerically trivial. Then there is $N\in\NN$ such that
     \[ \Gamma^{\circ N}=0\ \ \ \ \in A^n(X\times X)_{}\ .\]
\end{theorem}

 Actually, the nilpotence property (for all powers of $X$) could serve as an alternative definition of finite--dimensional motive, as shown by a result of Jannsen \cite[Corollary 3.9]{J4}.
   Conjecturally, all smooth projective varieties have finite--dimensional motive \cite{Kim}. We are still far from knowing this, but at least there are quite a few non--trivial examples:
 
\begin{remark} 
The following varieties have finite--dimensional motive: abelian varieties, varieties dominated by products of curves \cite{Kim}, $K3$ surfaces with Picard number $19$ or $20$ \cite{P}, surfaces not of general type with $p_g=0$ \cite[Theorem 2.11]{GP}, certain surfaces of general type with $p_g=0$ \cite{GP}, \cite{PW}, \cite{V8}, Hilbert schemes of surfaces known to have finite--dimensional motive \cite{CM}, generalized Kummer varieties \cite[Remark 2.9(\rom2)]{Xu} (an alternative proof is contained in \cite{FTV}),
 threefolds with nef tangent bundle \cite{Iy} (an alternative proof is given in \cite[Example 3.16]{V3}), fourfolds with nef tangent bundle \cite{Iy2}, certain threefolds of general type \cite[Section 8]{V5}, varieties of dimension $\le 3$ rationally dominated by products of curves \cite[Example 3.15]{V3}, varieties $X$ with $A^i_{AJ}(X)_{}=0$ for all $i$ \cite[Theorem 4]{V2}, products of varieties with finite--dimensional motive \cite{Kim}.
\end{remark}

\begin{remark}
It is an embarassing fact that up till now, all examples of finite-dimensional motives happen to lie in the tensor subcategory generated by Chow motives of curves, i.e. they are ``motives of abelian type'' in the sense of \cite{V3}. On the other hand, there exist many motives that lie outside this subcategory, e.g. the motive of a very general quintic hypersurface in $\PP^3$ \cite[7.6]{D}.
\end{remark}

\subsection{MCK decomposition}

\begin{definition}[Murre \cite{Mur}] Let $X$ be a projective quotient variety of dimension $n$. We say that $X$ has a {\em CK decomposition\/} if there exists a decomposition of the diagonal
   \[ \Delta_X= \pi_0+ \pi_1+\cdots +\pi_{2n}\ \ \ \hbox{in}\ A^n(X\times X)\ ,\]
  such that the $\pi_i$ are mutually orthogonal idempotents and $(\pi_i)_\ast H^\ast(X)= H^i(X)$.
  
  (NB: ``CK decomposition'' is shorthand for ``Chow--K\"unneth decomposition''.)
\end{definition}

\begin{remark} The existence of a CK decomposition for any smooth projective variety is part of Murre's conjectures \cite{Mur}, \cite{J2}. 
%If a quotient variety $X$
%has finite--dimensional motive, and the K\"unneth components are algebraic, then $X$ has a CK decomposition (this can be proven just as \cite{J2}, where this is stated for smooth $X$).
\end{remark}

\begin{definition}[Shen--Vial \cite{SV}] Let $X$ be a projective quotient variety of dimension $n$. Let $\Delta_X^{sm}\in A^{2n}(X\times X\times X)$ be the class of the small diagonal
  \[ \Delta_X^{sm}:=\bigl\{ (x,x,x)\ \vert\ x\in X\bigr\}\ \subset\ X\times X\times X\ .\]
  An MCK decomposition is a CK decomposition $\{\pi_i\}$ of $X$ that is {\em multiplicative\/}, i.e. it satisfies
  \[ \pi_k\circ \Delta_X^{sm}\circ (\pi_i\times \pi_j)=0\ \ \ \hbox{in}\ A^{2n}(X\times X\times X)\ \ \ \hbox{for\ all\ }i+j\not=k\ .\]
  
 (NB: ``MCK decomposition'' is shorthand for ``multiplicative Chow--K\"unneth decomposition''.) 
  \end{definition}
  
  \begin{remark} The small diagonal (seen as a correspondence from $X\times X$ to $X$) induces the {\em multiplication morphism\/}
    \[ \Delta_X^{sm}\colon\ \  h(X)\otimes h(X)\ \to\ h(X)\ \ \ \hbox{in}\ \MM_{\rm rat}\ .\]
 Suppose $X$ has a CK decomposition
  \[ h(X)=\bigoplus_{i=0}^{2n} h^i(X)\ \ \ \hbox{in}\ \MM_{\rm rat}\ .\]
  By definition, this decomposition is multiplicative if for any $i,j$ the composition
  \[ h^i(X)\otimes h^j(X)\ \to\ h(X)\otimes h(X)\ \xrightarrow{\Delta_X^{sm}}\ h(X)\ \ \ \hbox{in}\ \MM_{\rm rat}\]
  factors through $h^{i+j}(X)$.
  It follows that if $X$ has an MCK decomposition, then setting
    \[ A^i_{(j)}(X):= (\pi^X_{2i-j})_\ast A^i(X) \ ,\]
    one obtains a bigraded ring structure on the Chow ring: that is, the intersection product sends $A^i_{(j)}(X)\otimes A^{i^\prime}_{(j^\prime)}(X) $ to  $A^{i+i^\prime}_{(j+j^\prime)}(X)$.
      
  The property of having an MCK decomposition is severely restrictive, and is closely related to Beauville's ``weak splitting property'' \cite{Beau3}. For more ample discussion, and examples of varieties with an MCK decomposition, we refer to \cite[Section 8]{SV}, as well as \cite{V6}, \cite{SV2}, \cite{FTV}, \cite{EPW}.
    \end{remark}
    
\begin{theorem}[Vial \cite{V6}]\label{hilbk} Let $S$ be an algebraic $K3$ surface, and let $X=S^{[k]}$ be the Hilbert scheme of length $k$ subschemes of $S$. Then $X$ has a self--dual MCK decomposition.
\end{theorem}

\begin{proof} This is \cite[Theorem 1]{V6}. For later use, we briefly review the construction. First, one takes an MCK decomposition $\{\pi^S_i\}$ for $S$ (this exists, thanks to \cite{SV}). Taking products, this induces an MCK decomposition $\{ \pi^{S^r}_i\}$ for $S^r$, $r\in\NN$. This product MCK decomposition is invariant under the action of the symmetric group $\Sy_r$, and hence it induces an MCK decomposition
$\{ \pi_i^{S^{(r)}}\}$ for the symmetric products $S^{(r)}$, $r\in\NN$.
There is the isomorphism of de Cataldo--Migliorini \cite{CM}
  \[  \bigoplus_{\mu\in \Bb(k)} ({}^t \hat{\Gamma}_\mu)_\ast \colon\ \ \ A^i(X)\ \xrightarrow{\cong}\ \bigoplus_{\mu\in \Bb(k)}\ A^{i+l(\mu)-k}(S^{(\mu)})\ ,\]
  where $\Bb(k)$ is the set of partitions of $k$, $l(\mu)$ is the length of the partition $\mu$, and $S^{(\mu)}=S^{l(\mu)}/ {\Sy_{l(\mu)}}$, and ${}^t \hat{\Gamma}_\mu$ is a correspondence in $A^{k+l(\mu)}(S^{[k]}\times S^{(\mu)})$. Using this isomorphism, Vial defines \cite[Equation (4)]{V6} a natural CK decomposition for $X$, by setting
  \begin{equation}\label{defv} \pi_i^X:= {\displaystyle\sum_{\mu\in \Bb(k)}} \ {1\over m_\mu} \hat{\Gamma}_\mu \circ \pi^{S^{(\mu)}}_{i-2k+2l(\mu)}\circ {}^t \hat{\Gamma}_\mu\  , \end{equation}
  where the $m_\mu$ are rational numbers coming from the de Cataldo--Migliorini isomorphism.
  The $\{\pi^X_i\}$ of definition (\ref{defv}) are proven to be an MCK decomposition.
  
The self--duality of the $\{ \pi^X_i\}$ is apparent from definition (\ref{defv}).  
 
 \end{proof}

\begin{remark}\label{compat} It follows from definition (\ref{defv}) that the de Cataldo--Migliorini isomorphism is compatible with the bigrading of the Chow ring, in the sense that there are induced isomorphisms
  \[   \bigoplus_{\mu\in \Bb(k)} ({}^t \hat{\Gamma}_\mu)_\ast \colon\ \ \ A^i_{(j)}(X)\ \xrightarrow{\cong}\ \bigoplus_{\mu\in \Bb(k)}\ A^{i+l(\mu)-k}_{(j)}(S^{(\mu)})\ .\]
  In particular, there are split injections
   \[   \bigoplus_{\mu\in \Bb(k)} ({}^t {\Gamma}_\mu)_\ast \colon\ \ \ A^i_{(j)}(X)\ \xrightarrow{\cong}\ \bigoplus_{\mu\in \Bb(k)}\ A^{i+l(\mu)-k}_{(j)}(S^{\mu})\ .\]  
  \end{remark}

\begin{lemma}[Shen--Vial]\label{diag} Let $X$ be a projective quotient variety of dimension $n$, and suppose $X$ has a self--dual MCK decomposition. Then
  \[ \begin{split} &\Delta_X\ \ \in A^n_{(0)}(X\times X)\ ,\\
                         &\Delta_X^{sm}\ \ \in A^{2n}_{(0)}(X\times X\times X)\ .\\
                \end{split} \]        
\end{lemma}

\begin{proof} The first statement follows from \cite[Lemma 1.4]{SV2} when $X$ is smooth. The same argument works for projective quotient varieties; the point is just that
  \[ \begin{split} \Delta_X &= {\displaystyle\sum_{i=0}^{2n}}\pi_{i}^X =   {\displaystyle\sum_{i=0}^{2n}}\pi_{i}^X\circ \pi_i^X\\
                                        &=   {\displaystyle\sum_{i=0}^{2n}}({}^t\pi_{i}^X \times  \pi_i^X)_\ast \Delta_X\\
                                        &=  {\displaystyle\sum_{i=0}^{2n}}(\pi_{2n-i}^X \times  \pi_i^X)_\ast \Delta_X\\
                                        &=  (\pi_{2n}^{X\times X})_\ast \Delta_X\ \ \ \in A^n_{(0)}(X\times X)\ .\\
                                    \end{split}\]
                    (Here, the second line follows from Lieberman's lemma \cite[Lemma 3.3]{V3}, and the last line is the fact that the product of $2$ MCK decompositions is MCK.)
                    
           The second statement is proven for smooth $X$ in \cite[Proposition 8.4]{SV}; the same argument works for projective quotient varieties.                             
       \end{proof}

\subsection{Birational invariance}

\begin{proposition}[Rie\ss \cite{Rie}, Vial \cite{V6}]\label{birat} Let $X$ and $X^\prime$ be birational hyperk\"ahler varieties. Assume $X$ has an MCK decomposition. Then also $X^\prime$ has an MCK decomposition, and there are natural isomorphisms
  \[ A^i_{(j)}(X)\cong A^i_{(j)}(X^\prime)\ \ \ \hbox{for\ all\ }i,j\ .\]
\end{proposition}

\begin{proof} As noted by Vial \cite[Introduction]{V6}, this is a consequence of Rie\ss's result that $X$ and $X^\prime$ have isomorphic Chow motive (as algebras in the category of Chow motives). For more details, cf. \cite[Section 6]{SV} or \cite[Lemma 2.8]{BlochHK4}.
\end{proof}

\subsection{A commutativity lemma}

\begin{lemma}\label{comm} Let $S$ be an algebraic $K3$ surface, and let $\{\pi_i^S\}$ be the MCK decomposition as above. Let $h\in\aut(S)$.
Then
  \[ \Gamma_h\circ \pi_i^S=\pi_i^S\circ \Gamma_h\ \ \ \hbox{in}\ A^2(S\times S)\ \ \ \forall i\ .\]
\end{lemma}

\begin{proof} It suffices to prove this for $i=0$. Indeed, by definition of $\{\pi_i^S\}$ we have
  \[ \begin{split} \pi_4^S&:={}^t \pi_0^S\ \ \ \hbox{in}\ A^2(S\times S)\ ,\\
                      \pi_2^S&:=\Delta_S-\pi_0^S-\pi_4^S\ .\\
                      \end{split}\]
        Supposing the lemma holds for $i=0$, by taking transpose correspondences we get an equality
        \[ \Gamma_{h^{-1}}\circ \pi_4^S= \pi_4^S\circ \Gamma_{h^{-1}}\ \ \ \hbox{in}\ A^2(S\times S)\ .\]      
        Composing on both sides with $\Gamma_h$, we get
        \[ \pi_4^S\circ \Gamma_h=\Gamma_h\circ \pi_4^S\ \ \ \hbox{in}\ A^2(S\times S)\ .\]      
        Next, since obviously the diagonal $\Delta_S$ commutes with $\Gamma_h$, we also get
        \[  \Gamma_h\circ \pi_2^S=   \Gamma_h\circ (\Delta_S-\pi_0^S-\pi_4^S)=  (\Delta_S-\pi_0^S-\pi_4^S)\circ \Gamma_h=   \pi_2^S\circ \Gamma_h\ \ \ \hbox{in}\ A^2(S\times S)      \ .\]
        
     It remains to prove the lemma for $i=0$. The projector $\pi_0^S$ is defined as
     \[ \pi_0^S=\oo_S\times S\ \ \ \in A^2(S\times S)\ ,\]
     where $\oo_S\in A^2(S)$ is the ``distinguished point'' of \cite{BV} (any point lying on a rational curve in $S$ equals $\oo_S$ in $A^2(S)$). It is known \cite{BV} that          
     \[\ima\bigl( A^1(S)\otimes A^1(S)\ \to\ A^2(S)\bigr) =\QQ[\oo_S]   \ .\]
     It follows that there exist divisors $D_1, D_2\in A^1(S)$ such that $\oo_S=D_1\cdot D_2$, and so
     \[ h^\ast(\oo_S) =h^\ast(D_1\cdot D_2)  =h^\ast(D_1)\cdot h^\ast(D_2)\ \ \ \in \QQ[\oo_S]\ .\]
     Since $h^\ast(\oo_S)$ is the class of a point $h^{-1}(x)$ (where $x\in S$ is any point lying on a rational curve), it has degree $1$ and thus
     \[ h^\ast(\oo_S)=\oo_S\ \ \ \hbox{in}\ A^2(S)\ .\]
     
     Using Lieberman's lemma \cite[Lemma 3.3]{V6}, we find that
     \[  \begin{split} \pi_0^S\circ \Gamma_h &=  ({}^t \Gamma_h\times\Delta_S)_\ast (\pi_0^S)\\ &= ({}^t \Gamma_h\times\Delta_S)_\ast(\oo_S\times S) \\ &= h^\ast(\oo_S)\times S \\ &=\oo_S\times S=\pi_0^S\ \ \ \hbox{in}\ A^2(S\times S)\ ,\\
     \end{split}\]
     whereas obviously
     \[ \Gamma_h\circ \pi_0^S = (\Delta_S\times \Gamma_h)_\ast (\oo_S\times S)=\oo_S\times S=\pi_0^S\ \ \ \hbox{in}\ A^2(S\times S)\ .\]
     This proves the $i=0$ case of the lemma.
      \end{proof}
    
  The following lemmas establish some corollaries of lemma \ref{comm}:

      \begin{lemma}\label{idemp} Let $S$ be an algebraic $K3$ surface, and $G_S\subset\aut(S)$ a group of finite order $k$. For any $r\in\NN$, let $\{\pi_i^{S^r}\}$ denote the product MCK decomposition of $S^r$ induced by the MCK decomposition of $S$ as above. Let
  \[ \Delta^G_{S^r}:={1\over k}{\displaystyle\sum_{g\in G_S}}\ \Gamma_g\times\cdots\times\Gamma_g\ \ \ \in A^{2r}(S^r\times S^r)\ .\]
  Then
  \[ \Delta^G_{S^r}\circ \pi_i^{S^r}= \pi_i^{S^r}\circ \Delta^G_{S^r}\ \ \ \in A^{2r}(S^r\times S^r)\ \]
  is an idempotent, for any $i$.
  \end{lemma}
  
 \begin{proof} It suffices to prove the commutativity statement. (Indeed, since both $\Delta^G_{S^r}$ and $\pi_i^{S^r}$ are idempotent, the idempotence of their composition follows immediately from the stated commutativity relation.) 
  To prove the commutativity statement, we will prove more precisely that for any $h\in\aut(S)$ we have equality
    \begin{equation}\label{eachg}  \Gamma_{h^{\times r}}\circ \pi_i^{S^r}= \pi_i^{S^r} \circ \Gamma_{h^{\times r}}  \ \ \ \in A^{2r}(S^r\times S^r)\ .  \end{equation}
  This can be seen as follows:
  we have
    \[ \begin{split}  \Gamma_{h^{\times r}}\circ \pi_i^{S^r} &= (\Gamma_h\times \cdots \times\Gamma_h)\circ ({\displaystyle\sum_{i_1+\cdots+i_r=i}} \pi_{i_1}^S\times \cdots\times \pi_{i_r}^S)\\
                                                           &=   {\displaystyle\sum_{i_1+\cdots+i_r=i}} (\Gamma_h\circ \pi^S_{i_1})\times \cdots \times (\Gamma_h\circ \pi^S_{i_r})\\  
                                                           &=  {\displaystyle\sum_{i_1+\cdots+i_r=i}} ( \pi^S_{i_1}\circ \Gamma_h)\times \cdots \times ( \pi^S_{i_r}\circ \Gamma_h)\\  
                                                           &= {\displaystyle\sum_{i_1+\cdots+i_r=i}} (\pi_{i_1}^S\times \cdots\times \pi_{i_r}^S)      \circ (\Gamma_h\times\cdots\times\Gamma_h)\\
                                                           &= \pi_i^{S^r}\circ \Gamma_{h^{\times r}}\ \ \ \hbox{in}\ A^{2r}(S^r\times S^r)\ .\\
                                                           \end{split}\]
                Here, the first and last lines are the definition of the product MCK decomposition for $S^r$; the second and fourth line are just regrouping, and the third line is lemma \ref{comm}.                                           
       \end{proof}
  
 \begin{lemma}\label{idempx}  Let $S$ be an algebraic $K3$ surface, and $G_S\subset\aut(S)$ a group of finite order $k$. For any $r\in\NN$, let $X=S^{[r]}$ and let $G\subset\aut(X)$ be the group of natural automorphisms induced by $G_S$. Let $\{\pi_i^X\}$ be the MCK decomposition of theorem \ref{hilbk}. Let $\Delta^G_X$ denote the correspondence
   \[ \Delta^G_{X}:={1\over k}{\displaystyle\sum_{g\in G}}\ \Gamma_g\ \ \ \in A^{2r}(X\times X)\ .\]
  Then
  \[ \Delta^G_{X}\circ \pi_i^{X}= \pi_i^{X}\circ \Delta^G_{X}\ \ \ \in A^{2r}(X\times X)\ \]
  is an idempotent, for any $i$.
  \end{lemma}
  
 \begin{proof} Again, it suffices to prove the commutativity statement. This can be done as follows: for any $g\in G$, we can write $g=h^{[r]}$ where $h\in\aut(S)$. Then we have 
    \[ \begin{split}  \Gamma_g\circ \pi_i^X &= \Gamma_g\circ {\displaystyle\sum_{\mu\in \Bb(k)}} \ {1\over m_\mu} {\Gamma}_\mu \circ \pi^{S^{\mu}}_{i-2k+2l(\mu)}\circ {}^t {\Gamma}_\mu\\
                                   &=   {\displaystyle\sum_{\mu\in \Bb(k)}} \ {1\over m_\mu}\Gamma_g\circ {\Gamma}_\mu \circ \pi^{S^{\mu}}_{i-2k+2l(\mu)}\circ {}^t {\Gamma}_\mu\\
                                   &=  {\displaystyle\sum_{\mu\in \Bb(k)}} \ {1\over m_\mu} {\Gamma}_\mu \circ \Gamma_{h^{\times l(\mu)}} \circ \pi^{S^{\mu}}_{i-2k+2l(\mu)}\circ {}^t {\Gamma}_\mu\\
                                   & =  {\displaystyle\sum_{\mu\in \Bb(k)}} \ {1\over m_\mu} {\Gamma}_\mu \circ \pi^{S^{\mu}}_{i-2k+2l(\mu)}\circ \Gamma_{h^{\times l(\mu)}} \circ {}^t {\Gamma}_\mu\\
                                   & =  {\displaystyle\sum_{\mu\in \Bb(k)}} \ {1\over m_\mu} {\Gamma}_\mu \circ \pi^{S^{\mu}}_{i-2k+2l(\mu)}\circ {}^t {\Gamma}_\mu\circ \Gamma_g\\
                                   &= \pi_i^X\circ \Gamma_g\ \ \ \hbox{in}\ A^{2r}(X\times X)\ .\\
                                   \end{split}\]
                        Here, the first line follows from the definition of $\pi_i^X$ (definition (\ref{defv})). The second line is just regrouping, the third line is by construction of natural automorphisms of $X$, the fourth line is equality (\ref{eachg}) above, and the fifth line is again by construction of natural automorphisms.          
                                   \end{proof}

\begin{lemma}\label{quotientmck} Let $S$ be an algebraic $K3$ surface, and let $X=S^{[r]}$ be the Hilbert scheme of length $r$ subschemes. Let $G\subset\aut(X)$ a group of finite order $k$ of natural automorphisms. Then the quotient $Y:=X/G$ has a self--dual MCK decomposition.
\end{lemma}

\begin{proof} Let $p\colon X\to Y$ denote the quotient morphism. One defines
    \[ \pi_j^Y:= {1\over k} \Gamma_p\circ \pi_j^X\circ {}^t \Gamma_p\ \ \ \in A^{2r}(Y\times Y)\ ,\]
    where $\{\pi^X_j\}$ is the self--dual MCK decomposition of theorem \ref{hilbk}.
    This defines a self--dual CK decomposition $\{\pi_j^Y\}$, since
    \[   \begin{split} \pi_i^Y\circ \pi_j^Y &= {1\over k^2} \Gamma_p\circ \pi_i^X\circ {}^t \Gamma_p  \circ   \Gamma_p\circ \pi_j^X\circ {}^t \Gamma_p  \\
                         &= {1\over k} \Gamma_p\circ \pi_i^X\circ \Delta^G_X \circ \pi_j^X\circ {}^t \Gamma_p \\
                         &= {1\over k} \Gamma_p\circ \pi_i^X\circ \pi_j^X\circ \Delta^G_X\circ {}^t \Gamma_p \\
                         &=\begin{cases}  0 &\hbox{if\ }i\not=j\ ;\\
                                                    {1\over k} \Gamma_p\circ \pi_i^X\circ {}^t \Gamma_p=\pi_i^Y &\hbox{if\ }i=j\ .\\
                                                   \end{cases} 
                                              \end{split}\]   
           (Here, in the third line we have used lemma \ref{comm}.)
           
   It remains to check this CK decomposition is multiplicative. To this end, let $i,j,k$ be integers with $k\not=i+j$. We note that                                 
           \[    \begin{split}    
      \pi^Y_k\circ \Delta_Y^{sm}\circ (\pi^Y_i\times \pi^Y_j)  &= {1\over k^3}\ \ \Gamma_p\circ \pi^X_k\circ {}^t \Gamma_p\circ \Delta^Y_{sm}\circ \Gamma_{p\times p}\circ (\pi^X_i\times \pi^X_j)\circ {}^t \Gamma_{p\times p}\\
                          &=\Gamma_p\circ \pi^X_k\circ \Delta^G_X\circ \Delta_X^{sm}\circ (\Delta^G_X\times \Delta^G_X)\circ   (\pi^X_i\times \pi^X_j)\circ {}^t \Gamma_{p\times p}\\
                       & = \Gamma_p\circ \Delta^G_X\circ \pi^X_k\circ \Delta_X^{sm}\circ   (\pi^X_i\times \pi^X_j)\circ (\Delta^G_X\times \Delta^G_X)\circ {}^t \Gamma_{p\times p}\\
                       & =0\ \ \ \hbox{in}\ A^{2n}(Y\times Y\times Y)\ .\\
                       \end{split}\]
                 Here, the first equality is by definition of the $\pi^Y_i$, the second equality is lemma \ref{sm} below, the third equality follows from lemma \ref{idempx}, and the fourth equality is the fact that the $\pi^X_i$ are an MCK decomposition for $X$.      
        
     \begin{lemma}\label{sm} There is equality
     \[ \begin{split} {}^t \Gamma_p\circ \Delta_Y^{sm}\circ \Gamma_{p\times p}&= (\sum_{g\in G} \Gamma_g)\circ \Delta_X^{sm}\circ \bigl((\sum_{g\in G} \Gamma_g)\times 
         (\sum_{g\in G}           \Gamma_g)\bigr) \\  
         &= k^3\ \Delta^G_X\circ \Delta_X^{sm}\circ (\Delta^G_X\times \Delta^G_X)\ \ \ \hbox{in}\ A^{2n}(X\times X\times X)\ .\\
         \end{split}\]                    
                                                \end{lemma}
                                                
  \begin{proof} The second equality is just the definition of $\Delta^G_X$. As to the first equality, we first note that
         \[  \Delta_Y^{sm}  =(p\times p\times p)_\ast (\Delta_X^{sm}) = \Gamma_p\circ \Delta_X^{sm}\circ {}^t \Gamma_{p\times p}\ \ \ \hbox{in}\ A^{3n}(Y\times Y\times Y)\ .\]
                 This implies that
                 \[ {}^t \Gamma_p\circ \Delta_Y^{sm}\circ \Gamma_{p\times p}= {}^t \Gamma_p\circ \Gamma_p\circ \Delta_X^{sm}\circ {}^t \Gamma_{p\times p}\circ \Gamma_{p\times p}\ .\]
                 But ${}^t \Gamma_p\circ \Gamma_p=\sum_{g\in G} \Gamma_g$, and thus
                 \[ {}^t \Gamma_p\circ \Delta_Y^{sm}\circ \Gamma_{p\times p}=  (\sum_{g\in G} \Gamma_g)\circ \Delta_X^{sm}\circ \bigl((\sum_{g\in G} \Gamma_g)\times 
         (\sum_{g\in G}           \Gamma_g)\bigr) \ \ \ \hbox{in}\ A^{2n}(X\times X\times X)\ ,\]
         as claimed. 
         \end{proof}

\end{proof}

 \subsection{An injectivity result}
 
 \begin{lemma}[Vial \cite{V6}]\label{inj1} Let $S$ be an algebraic $K3$ surface, and $X=S^{[r]}$ the Hilbert scheme of length $r$ subschemes of $S$. The cycle class map induces a map
   \[ A^i_{(0)}(X)\ \to\ H^{2i}(X) \]
   that is injective for $i\ge 2r-1$.
 \end{lemma} 
 
 \begin{proof} This is stated without proof in \cite[Introduction]{V6}. The idea is as follows: let $i\ge 2r-1$. Using remark \ref{compat}, we obtain a commutative diagram
   \[  \begin{array}[c]{ccc} A^i_{(0)}(X) &\to& A^i_{(0)}(S^r)\\
                        \downarrow &&\downarrow\\
                        H^{2i}(X) &\to& \ H^{2i}(S^r)\ ,\\
                        \end{array}\]
  where horizontal arrows are split injections, and vertical arrows are restrictions of the cycle class map. It thus suffices to prove that restriction of the cycle class map
  \[ A^i_{(0)}(S^r)\ \to\ H^{2i}(S^r) \]
  is injective. 
  
  Let $\{\pi_j^{S^r}\}$ denote the product MCK decomposition constructed above.
  It follows from the definition of $A^i_{(0)}(S^r)$ that 
    \[ (\pi_{2i}^{S^r})_\ast=\ide\colon\ \ \ A^i_{(0)}(S^r)\ \to\ A^i(S^r)\ .\] 
    Let $x\in S$ be a point such that $x=\oo_S$ in $A^2(S)$. Then the projector $\pi^{S^r}_{4r}$ is supported on $S^r\times (x\times\cdots\times x)$, and $\pi_{4r-2}^{S^r}$ is supported on
    \[ S^r\times (S\times x\times\cdots\times x)\cup S^r\times (x\times S\times x\times\cdots\times x)\cup\cdots\cup S^r\times (x\times\cdots\times x\times S)\ \ \ \subset S^r\times S^r\ .\]
    It follows that for $i=2r$ there is a factorization
    \[ \begin{array}[c]{ccc}
               A^{2r}_{(0)}(S^r) &\to& H^{4r}(S^r)\\
               \downarrow&&\downarrow\\
               A^0(x\times\cdots\times x) &\to& H^0(x\times\cdots\times x)\\
               \downarrow&&\downarrow\\
               A^{2r}_{(0)}(S^r) &\to& \ \ H^{4r}(S^r)\ ,\\
               \end{array}\]
               where composition of vertical arrows is $(\pi_{4r}^{S^r})_\ast=\ide$. This implies $A^{2r}_{(0)}(S^r)\cong\QQ$ and the map to $H^{4r}(S^r)$ is an isomorphism.
               
               Likewise, for $i=2r-1$ there is a factorization
           \[     \begin{array}[c]{ccc}
               A^{2r-1}_{(0)}(S^r) &\to& H^{4r-2}(S^r)\\
               \downarrow&&\downarrow\\
             \bigoplus  A^1(S) &\to& \bigoplus H^2(S)\\
               \downarrow&&\downarrow\\
               A^{2r-1}_{(0)}(S^r) &\to& \ \ H^{4r-2}(S^r)\ ,\\
               \end{array}\]     
               where composition of vertical maps is $(\pi_{4r-2}^{S^r})_\ast=\ide$. Since the middle horizontal arrow is injective, this implies the other horizontal arrows are injective as well.
               \end{proof}
               
    \begin{remark} As explained in \cite{SV}, conjecturally the restriction of the cycle class map
    \[ A^i_{(0)}(X)\ \to\ H^{2i}(X) \]
    is injective for any variety $X$ having an MCK decomposition. This is related to Murre's ``conjecture D'' \cite{Mur}, and the expectation that the bigrading $A^\ast_{(\ast)}$ should give a splitting of a Bloch--Beilinson filtration.
    
    As we will see below (lemma \ref{inj2}), for Hilbert schemes of special $K3$ surfaces one can prove more than lemma \ref{inj1}.
    \end{remark}

\subsection{LSY surfaces}

\begin{definition}\label{lsy} An {\em LSY surface\/} (short for ``Livn\'e--Sch\"utt--Yui surface'') is a projective $K3$ surface $S$, with the following properties:

\noindent
(\rom1)
There is a group
  $ G_S\subset\aut(S)$
  acting trivially on $NS(S)$;
  
\noindent
(\rom2) Let $k:=\ord(G_S)$. There is equality
  \[ \dim(T_S)=\phi(k)\ ,\]
  where $T_S\subset H^2(S)$ denotes the transcendental lattice, and $\phi(k)$ is Euler's totient function.
  \end{definition}
  
  \begin{remark} Assumption (\rom1) of definition \ref{lsy} implies that $G_S$ is a finite cyclic group \cite{LSY}, so the definition of the integer $k$ makes sense. 
  Under assumption (\rom1), $\phi(k)$ divides $\dim(T_S)$, so assumption (\rom2) is equivalent to asking that the Picard number of $S$ is maximal among all $K3$ surfaces satisfying (\rom1) for a given value of $k=ord(G_S)$.
  \end{remark}
  
\begin{theorem}[Livn\'e--Sch\"utt--Yui \cite{LSY}]\label{lsythm} Let $S$ be an LSY surface, and $k:=\ord(G_S)$. Then
  \[ k\in\ \ \Bigl\{ 3, 5, 7, 9, 11, 12, 13, 17, 19, 25, 27, 28, 36, 42, 44, 66\Bigr\}\ .\]
  Conversely, for each of these values of $k$, there exists a unique LSY surface $S_k$ with $k:=\ord(G_S)$ up to isomorphism.
  All these surfaces $S_k$ have finite--dimensional motive. 
  \end{theorem}
  
 \begin{proof} This is \cite[Theorems 1 and 2]{LSY}, combined with the explicit descriptions given in \cite[Sections 3 and 4]{LSY}. 
 \end{proof}

\begin{remark} The study of LSY surfaces was initiated by Vorontsov \cite{Vor} and Kondo \cite{Ko}.  Livn\'e--Sch\"utt--Yui give explicit equations for all the surfaces $S_k$ \cite[Sections 3 and 4]{LSY}. 
To give one example, the surface $S_{66}$ can be described as a hypersurface of degree $12$ 
   \[ x_0^2 + x_1^3 + x_2^{11}x_3 + x_3^{12}=0  \]
   in a weighted projective space $\PP(6,4,1,1)$. (As explained in \cite[Remark 2]{LSY}, the surface $S_{66}$ can also be described as an elliptic surface.). 

With the exception of $S_3$ (which is of maximal Picard rank $\rho(S_3)=20$), all the $S_k$ are Delsarte surfaces; as such, they are dominated by Fermat surfaces. This immediately implies finite--dimensionality of the $S_k$.
\end{remark}

\subsection{Sch\"utt surfaces}

\begin{definition}\label{schu} A {\em Sch\"utt surface\/} is a projective $K3$ surface $S$, with the following properties:

\noindent
(\rom1)
There is a group
  $ G_S\subset\aut(S)$  acting trivially on $NS(S)$;
  
\noindent
(\rom2) The order $k:=\ord(G_S)$ is a $2$--power;

\noindent
(\rom3) There is equality
  \[ \dim(T_S)=k\ ,\]
  where $T_S\subset H^2(S)$ denotes the transcendental lattice.
  \end{definition}

Complementing results of \cite{Vor}, \cite{Ko}, \cite{LSY}, Sch\"utt has classified Sch\"utt surfaces:

\begin{theorem}[Sch\"utt \cite{Schu}]\label{thschu} Let $S$ be a Sch\"utt surface, and $k=\ord{G_S}$. Then
  \[ k\in \bigl\{ 2,4,8,16\bigr\}\ .\]
  Conversely:
  
  \begin{description}
   \item[$k=2$] there exists a unique Sch\"utt surface $S_2$ with $k=2$ (up to isomorphism);
   \item[$k=4$] any Sch\"utt surface with $k=4$ is an element of the one--dimensional family $S_{4,\lambda}^{unimod}$  ($\lambda\in\C$) or the one--dimensional family $S_{4,\lambda}^{non}$ ($\lambda\in\C$);
   \item[$k=8$] any Sch\"utt surface with $k=8$ is an element of a one--dimensional family $S_{8,\lambda}$ ($\lambda\in\C$);
   \item[$k=16$] any Sch\"utt surface with $k=16$ is an element of a one--dimensional family $S_{16,\lambda}$ ($\lambda\in\C$).
  \end{description}
  \end{theorem}
  
\begin{proof} This is \cite[Theorem 1]{Schu}.
\end{proof}

\begin{remark} The surfaces in theorem \ref{thschu} are given by explicit equations. For example, the family $S_{4,\lambda}^{unimod}$ is defined by the Weierstrass equation
  \[ y^2=x^3-3\lambda t^4x+t^5+t^7\ \]
  \cite[Theorem 1]{Schu}.
  For a generic $\lambda$, this surface will have $\rank(T_S)=4$, and so the surface is a Sch\"utt surface.
  \end{remark}

Contrary to the LSY surfaces, {\em not\/} all Sch\"utt surfaces have provably finite--dimensional motive. Some of them do, however:

\begin{proposition}[Sch\"utt \cite{Schu}]\label{schuttfdim} Let $S$ be either a $S_{4,\lambda}^{unimod}$ with $\lambda$ generic, or 
  \[ S\in \Bigl\{ S_2, S_{4,0}^{non}, S_{8,0}, S_{8,2}, S_{8,\sqrt{3}}, S_{16,0}, S_{16,2}, S_{16,\sqrt{3}}\Bigr\}\ .\]
  Then $S$ is a Sch\"utt surface with finite--dimensional motive.
  \end{proposition}
 
 \begin{proof} A generic element of the pencil $S_{4,\lambda}^{unimod}$ is a Sch\"utt surface \cite{Schu}. It also has a Shioda--Inose structure \cite{Schu}, which implies finite--dimensionality. The surface $S_2$ has Picard number $20$, hence is Kummer. The other surfaces in proposition \ref{schuttfdim} are dominated by Fermat surfaces \cite[Lemma 18]{Schu}, hence have finite--dimensional motive.
 \end{proof}

\subsection{Transcendental part of the motive}

\begin{theorem}[Kahn--Murre--Pedrini \cite{KMP}]\label{kmp} Let $S$ be a surface. There exists a decomposition
  \[ h_2(S)= t_2(S)\oplus h_2^{alg}(S)\ \in \MM_{\rm rat}\ ,\]
  such that
  \[  H^\ast(t_2(S),\QQ)= H^2_{tr}(S)\ ,\ \ H^\ast(h_2^{alg}(S),\QQ)=NS(S)_{\QQ}\ \]
  (here $H^2_{tr}(S)$ is defined as the orthogonal complement of the N\'eron--severi group $NS(S)_{\QQ}$ in $H^2(S,\QQ)$),
  and
   \[ A^\ast(t_2(S))_{\QQ}=A^2_{AJ}(S)_{\QQ}\ .\]
   (The motive $t_2(S)$ is called the {\em transcendental part of the motive\/}.)
 
   Let $h_2^{alg}(S)=(S,\pi_2^{alg},0)\in\MM_{\rm rat}$. The projector $\pi_2^{alg}$ is supported on $D\times D$, for $D\subset S$ a divisor.
    \end{theorem}

\subsection{Natural automorphisms of Hilbert schemes}

\begin{definition}[Boissi\`ere \cite{Bo}] Let $S$ be a surface, and let $X=S^{[k]}$ denote the Hilbert scheme of length $k$ subschemes. An automorphism $\psi\in\aut(S)$ induces an automorphism $\psi^{[k]}$ of $X$. This determines a homomorphism
  \[ \begin{split} \aut(S)\ &\to\ \aut(X)\ ,\\
              \psi\ &\mapsto\ \psi^{[k]}\ ,\\
           \end{split}\]
    which is injective \cite{Bo}. The image of this homomorphism is called the group of {\em natural automorphisms\/} of $X$.          
\end{definition}

\begin{theorem}[Boissi\`ere--Sarti \cite{BoSa}] Let $S$ be a $K3$ surface, and $X=S^{[k]}$. Let $E\subset X$ denote the exceptional divisor of the Hilbert--Chow morphism. An automorphism $g\in\aut(X)$ is natural if and only if $g^\ast(E)=E$ in $NS(X)$.
\end{theorem}

\begin{proof} This is \cite[Theorem 1]{BoSa}.
\end{proof}

\begin{remark} To find examples of non--natural automorphisms of a Hilbert scheme $X$, Boissi\`ere and Sarti introduce the notion of {\em index\/} of an automorphism of $X$. For Hilbert schemes of a generic algebraic $K3$ surface, the index of an automorphism is $1$ if and only if the automorphism is natural \cite[section 4]{BoSa}.
\end{remark}

\subsection{A support lemma}

For later use, we establish a lemma:

  \begin{lemma}\label{pi2k} Let $S$ be an LSY surface or Sch\"utt surface, and let $G_S$ be the order $k$ group as in definition \ref{lsy}, resp. definition \ref{schu}. For any $r\in\NN$ let
  \[    \Delta^G_{S^r}:= {1\over k}{\displaystyle\sum_{g\in G_{S}}} \Gamma_g\times\cdots\times\Gamma_g\in A^{2r}(S^r\times S^r)\ .\]  
  Let $\{\pi_j^{S^r}\}$ denote the product MCK decomposition for $S^r$ as above. There is a homological equivalence
    \[  \Delta^G_{S^r}\circ \pi_2^{S^r} = \gamma\ \ \ \ \hbox{in}\ H^{4r}(S^r\times S^r)   \ ,\]
    where $\gamma$ is a cycle supported on $C\times D\subset S^r\times S^r$, and $C\subset S^r$ is a curve and $D\subset S^r$ is a divisor.
    \end{lemma}
    
    \begin{proof} Let us first do the $r=1$ case. Since the group $G_S\subset\aut(S)$ consists of non--symplectic automorphisms, we have
     \[ (\Delta^{G}_S)_\ast=0\colon\ \ \ H^{2,0}(S)\ \to\ H^{2,0}(S)\ .\]
     Let $T\subset H^2(S)$ denote the transcendental lattice. Since $T$ defines an indecomposable Hodge structure (i.e., every Hodge sub--structure of $T$ is either $T$ or $0$), we must have
     \[ (\Delta^{G}_S)_\ast=0\colon\ \ \ T\ \to\ T\ .\]
    Since $\Delta^G_S$ acts as the identity on $NS(S)$, this implies
     \[ \Delta^G_S\circ \pi_2^S = \pi^{S,alg}_{2}\ \ \ \hbox{in}\ H^4(S\times S)\ .\]
     But $\pi^{S,alg}_2$ is supported on divisor times divisor (theorem \ref{kmp}); this proves the case $k=1$.
     
     For arbitrary $r$, note that (by definition of the product MCK decomposition)
     \[ \pi_2^{S^r}= \pi_2^S\times \pi_0^S\times\cdots\times \pi_0^S + \cdots +  \pi_0^S\times\cdots\times \pi_0^S\times \pi_2^S\ \ \ \in A^{2r}(S^r\times S^r)\ .\]
     Thus,
       \[ \begin{split}       \Delta^G_{S^r}\circ \pi_2^{S^r} &= {1\over k}{\displaystyle\sum_{h\in G_S}}\ (\Gamma_h\times\cdots\times\Gamma_h)\circ      ( \pi_2^S\times \pi_0^S\times\cdots\times \pi_0^S + \cdots +  \pi_0^S\times\cdots\times \pi_0^S\times \pi_2^S)\\
                                         &= {1\over k} {\displaystyle\sum_{h\in G_S}}\ (\Gamma_h\circ \pi_2^S)\times (\Gamma_h\circ \pi_0^S)\times\cdots\times (\Gamma_h\circ \pi_0^S) \\
                                            & \ \ \ \ \ \ \ \ \ \  \ \ \ \ \ \ \ \   +\cdots +
                                                  (\Gamma_h\circ \pi_0^S)\times\cdots\times (\Gamma_h\circ \pi_0^S)  \times    (\Gamma_h\circ \pi_2^S)\\
                                          &= {1\over k}{\displaystyle\sum_{h\in G_S}} (\Gamma_h\circ \pi_2^S) \times \pi_0^S \times \cdots \times  \pi_0^S +\cdots +  \pi_0^S \times \cdots \times  \pi_0^S \times 
                                             (\Gamma_h\circ \pi_2^S)\\
                                          &= (\Delta^G_S\circ \pi_2^S)\times \pi_0^S\times\cdots\times \pi_0^S +\cdots +  \pi_0^S\times\cdots\times \pi_0^S\times (\Delta^G_S\circ \pi_2^S)\\
                                          &= \pi_2^{S,alg}\times  \pi_0^S\times\cdots\times \pi_0^S +\cdots +  \pi_0^S\times\cdots\times \pi_0^S\times  \pi_2^{S,alg}\ \ \ \hbox{in}\ H^{4r}(S^r\times S^r)\ .
                                          \end{split}\]
                            Here, the second line is because $\Gamma_h\circ \pi_0^S=\pi_0^S$ (proof of lemma \ref{idemp}), and the last line is the $r=1$ case treated above. The last line is clearly a cycle supported on curve times divisor, and so the lemma is proven.
                           \end{proof}

\section{Main result}

\begin{theorem}\label{main3} Let $S_3$ be as in theorem \ref{lsythm}, and let $X$ be the Hilbert scheme $X=(S_3)^{[3]}$. Let $G\subset\aut(X)$ be the group of natural automorphisms induced by the order $3$ cyclic group $G_{S_3}\subset\aut(S_3)$ of definition \ref{lsy}. Then
 \[ \begin{split}  (\Delta^G_X)_\ast &=0\colon\ \ A^i_{(j)}(X)\ \to\ A^i(X)\ \ \hbox{for}\ (i,j)\in \Bigl\{ (2,2), (4,4), (3,2),(6,2),(6,4),(5,2)\Bigr\}\ .\\
                               %     (\Delta_X^G)_\ast &=\ide\colon\ \ A^6_{(j)}(X)\ \to\ A^6(X)\ \ \hbox{for}\ j\in   \{0,6\}    \ .\\ 
                               \end{split}\]
% Moreover, 
% \[ \begin{split} (\Delta_X^G)_\ast &=0\colon\ \ B^4_{(4)}(X)\ \to\ B^4(X)\ ,\\
%   \[ (\Delta_X^G)_\ast =\ide\colon\ \ B^5_{(4)}(X)\ \to\ B^5(X)\ .    
 %   \]
\end{theorem}

\begin{proof} In the course of this proof, let us write $S$ instead of $S_3$. The idea is to reduce to the action of automorphisms on $A^i(S^3)$ and $A^i(S^2)$ and $A^i(S)$. This reduction is possible thanks to the commutative diagram
  \begin{equation}\label{comdiag3}\begin{array} [c]{ccccccc}
        A^i_{(j)}(X) & \hookrightarrow & A^i_{(j)}(S^3)&\oplus & A^{i-1}_{(j)}(S^2)&\oplus & A^{i-2}_{(j)}(S)\\
          &&&&&&\\
        \ \ \ \  \downarrow {\scriptstyle (\Delta_X^G)_\ast} &&\ \ \ \downarrow {\scriptstyle (\Delta^G_{S^3})_\ast} && \ \ \ \downarrow {\scriptstyle (\Delta^G_{S^2})_\ast}&&\ \ \  \downarrow {\scriptstyle (\Delta^G_{S})_\ast}\\
         &&&&&&\\
         A^i_{(j)}(X) & \hookrightarrow & A^i_{(j)}(S^3)&\oplus & A^{i-1}_{(j)}(S^2)&\oplus & A^{i-2}_{(j)}(S)\\
         \end{array}\end{equation}
         
   Here, $\Delta^G_{S^r}$ is as in lemma \ref{pi2k}. This diagram commutes because of the construction of natural automorphisms. Horizontal arrows are injective because of remark \ref{compat}.

  To handle the action of $\Delta^G_{S^r}$ on $A^i_{(j)}(S^r)$ for $r=1,2,3$, we establish two lemmas:
  
  \begin{lemma}\label{pi2} There are homological equivalences
    \[  \begin{split}  \Delta^G_{S^3}\circ \pi_2^{S^3} = \gamma_2^{S^3}\ \ \ \ &\hbox{in}\ H^{12}(S^3\times S^3)\ ,\\
                            \Delta^G_{S^2}\circ \pi_2^{S^2} = \gamma_2^{S^2}\ \ \ \ &\hbox{in}\ H^{8}(S^2\times S^2)\ ,\\   
                           \Delta^G_{S}\circ \pi_2^{S} = \gamma_2^{S}\ \ \ \ &\hbox{in}\ H^{4}(S\times S)\ ,\\
                           \end{split}\]
              where $\gamma_2^{S^3}$ (resp. $\gamma_2^{S^2}$ resp. $\gamma_2^S$) is a cycle in 
          \[ \begin{split} &\ima\Bigl(  A_6(V_{2,3}\times W_{2,3})\ \to\ A^6(S^3\times S^3)\Bigr)\ \\    
             \hbox{(resp.\ }   &\ima\Bigl(  A_4(V_{2,2}\times W_{2,2})\ \to\ A^4(S^2\times S^2)\Bigr)\ ,\\    
              \hbox{(resp.\ }   &\ima\Bigl(  A_2(V_{2,1}\times W_{2,1})\ \to\ A^2(S\times S)\Bigr)\ \hbox{\ )},\\  
              \end{split}\]  
   and $V_{2,r}\subset S^r$ is a closed subvariety of codimension $2r-1$, and $W_{2,r}\subset S^r$ is closed of codimension $1$.          
   \end{lemma}  
   
   \begin{proof} This is a special case of lemma \ref{pi2k}.
   \end{proof}

   \begin{lemma}\label{pi4} There are homological equivalences
    \[  \begin{split}  \Delta^G_{S^3}\circ \pi_4^{S^3} = \gamma_4^{S^3}\ \ \ \ &\hbox{in}\ H^{12}(S^3\times S^3)\ ,\\
                            \Delta^G_{S^2}\circ \pi_4^{S^2} = \gamma_4^{S^2}\ \ \ \ &\hbox{in}\ H^{8}(S^2\times S^2)\ ,\\   
                %           \Delta^G_{S}\circ \pi_2^{S} = \gamma_2^{S}\ \ \ \ &\hbox{in}\ H^{4}(S\times S)\ ,\\
                           \end{split}\]
              where $\gamma_4^{S^3}$ (resp. $\gamma_4^{S^2}$) is a cycle in 
          \[ \begin{split} &\ima\Bigl(  A_6(V_{4,3}\times W_{4,3})\ \to\ A^6(S^3\times S^3)\Bigr)\ \\    
             \hbox{(resp.\ }   &\ima\Bigl(  A_4(V_{4,2}\times W_{4,2})  \ \to\ A^4(S^2\times S^2)\Bigr)\ ,\\    
           %   \hbox{(resp.\ }   &\ima\Bigl(  A_2(V_{2,1}\times W_{2,1})\ \to\ A^2(S\times S)\ \hbox{)},\\  
              \end{split}\]  
   and $V_{4,3}, W_{4,3}\subset S^3$ are closed subvarieties of codimension $4$ resp. $2$, and $V_{4,2}, W_{4,2}\subset S^2$ are closed subvarieties of codimension $2$.          
   \end{lemma}     
   
   \begin{proof} Here we will use the fact that $S=S_3$ is $\rho$--maximal (i.e. the Picard number $\rho(S_3)$ is $20$). This means that the transcendental lattice $T\subset H^2(S)$ has rank $2$
   and injects (under the natural map $H^2(S)\to H^2(S,\C)$) into $H^{2,0}\oplus H^{0,2}$.
   It follows that (under the natural map $H^2(S)\to H^2(S,\C)$)
     \[ T\otimes T\ \subset\ H^{4,0}(S^2)\oplus H^{2,2}(S^2)\oplus H^{0,4}(S^2)\ .\]
     Let $h\in G_S$ be a generator. Since $h$ is non--symplectic, $h^\ast$ acts on $H^{2,0}$ as multiplication by a primitive $3$rd root of unity $\nu$. It follows that
     \[ (h\times h)^\ast=\nu^2\cdot\ide\colon\ \ \ H^{4,0}(S^2)\ \to\ H^{4,0}(S^2)\ ,\]
     and hence (since $\nu^2\not=1$)
     \[ (\Delta^G_{S^2})_\ast=0\colon\ \ \ H^{4,0}(S^2)\ \to\ H^{4,0}(S^2)\ .\]
     For the same reason, we also have
      \[ (\Delta^G_{S^2})_\ast=0\colon\ \ \ H^{0,4}(S^2)\ \to\ H^{0,4}(S^2)\ .\]  
      It follows that
      \[  (\Delta^G_{S^2})_\ast (T\otimes T) =   (\Delta^G_{S^2})_\ast \bigl( (T\otimes T)\cap F^2\bigr)\ \ \ \subset\ H^4(S^2)\ \]
      (here $F^\ast$ denotes the Hodge filtration on $H^\ast(-,\C)$).
      But $H^4(S^2)\cap F^2$ is generated by codimension $2$ cycles (indeed, $S$ is a Kummer surface, and so the Hodge conjecture is true for $S^r$ since it is true for self--products of abelian surfaces \cite[7.2.2]{Ab}). This means that there exist a codimension $2$ subvariety $V\subset S^2$ and a cycle $\gamma$ supported on $V\times V$ such that
      \[ \Delta^G_{S^2}\circ (\pi_2^{S,tr}\times \pi_2^{S,tr})-\gamma =0\ \ \ \hbox{in}\ H^{8}(S^2\times S^2)\ .\] 
      Next, let us write
      \[ H^2(S)=T\oplus N\ ,\]
      where $N:=NS(S)$.
      The action of $\Delta^G_{S^2}$ on $T\otimes N$ and on $N\otimes T$ is $0$. Indeed,
      \[ (h\times h)^\ast =\nu\cdot \ide\times \ide\colon\ \ \ T\otimes N\ \to\ T\otimes N\ ,\]
      and so
      \[ (\Delta^G_{S^2})_\ast = (\Delta^G_S\times\Delta_S)_\ast=0\colon\ \ \ T\otimes N\ \to\ T\otimes N\ .\]
      This means that
      \[ \Delta^G_{S^2}\circ (\pi_2^{S,tr}\times \pi_2^{S,alg}) =\Delta^G_{S^2}\circ ( (\pi_2^{S,alg}\times \pi_2^{S,tr})    =0\ \ \ \hbox{in}\ H^8(S^2\times S^2)\ .\]   
      The correspondences $\pi_0^S\times \pi_4^S$ and $\pi_4^S\times \pi_0^S$ are obviously supported on $V\times V\subset S^2\times S^2$ for some codimension $2$ subvariety $V\subset S^2$.
      It follows that
      \[  \Delta^G_{S^2}\circ \pi_4^{S^2} = \Delta^G_{S^2}\circ (\pi_2^{S,tr}\times \pi_2^{S,tr} + \pi_2^{S,alg}\times \pi_2^{S,alg} + \pi_0^S\times \pi_4^S + \pi_4^S\times \pi_0^S) =\gamma^\prime\ \ \ \hbox{in}\ H^8(S^2\times S^2)\ ,\]
      where $\gamma^\prime$ is supported on $V\times V\subset S^2\times S^2$, for $V\subset S^2$ of codimension $2$.      
       This proves the statement for $S^2$.
       
       The statement for $S^3$ follows immediately. Indeed, we have
       \[ \pi_4^{S^3} = \pi_0^S\times \pi_4^{S^2} + \pi_4^{S^2}\times \pi_0^S + \pi_4^{S^2}\times \pi_0^S\ \ \ \hbox{in}\ A^6(S^3\times S^3)\ ,\]
       where $\pi_0^S$ in the first (resp. second, resp. third) factor lies in the first (resp. second, resp. third) copy of $S$.
       But $\Gamma_h\circ \pi_0^S=\pi_0^S$ (proof of lemma \ref{idemp}), and so
       \[  \Delta^G_{S^3}\circ ( \pi_0^S\times \pi_4^{S^2}) = \pi_0^S\times (\Delta^G_{S^2}\circ \pi_4^{S^2})\ \ \ \hbox{in}\ A^6(S^3\times S^3)\ ,\]
       which (by the above) is homologically supported on $V_{4,3}\times W_{4,3}\subset S^3\times S^3$, where codim. $V_{4,3}=4$, codim. $W_{4,3}=2$.       
       \end{proof}

 We are now in position to wrap up the proof of theorem \ref{main3}. 
 Let us first consider $0$--cycles, i.e. $i=6$. The commutative diagram (\ref{comdiag3}) simplifies to
   \begin{equation}\label{diagsimp} \begin{array}[c]{ccc}
       A^6_{(j)}(X) & \hookrightarrow & A^6_{(j)}(S^3)\\
        &&\\
          \ \ \ \  \downarrow {\scriptstyle (\Delta_X^G)_\ast} &&\ \ \ \downarrow {\scriptstyle (\Delta^G_{S^3})_\ast} \\
          &&\\
     A^6_{(j)}(X) & \hookrightarrow & A^6_{(j)}(S^3)\\
     \end{array}\end{equation}
  In case $0<j<6$ (i.e. $j=2$ or $4$), we need to prove that
  \[  (\Delta^G_X)_\ast A^6_{(j)}(X) = 0\ ,\]
  which (in view of the above diagram) reduces to proving that
  \begin{equation}\label{0cy} (\Delta^G_{S^3})_\ast   A^6_{(j)}(S^3) = (\Delta^G_{S^3}\circ \pi^{S^3}_{12-j})_\ast A^6(S^3)=0\ \ \ \hbox{for\ }j=2,4\ .\end{equation}
  In view of lemma \ref{comm}, we have
  \[  \Delta^G_{S^3}\circ \pi^{S^3}_{12-j} = \pi^{S^3}_{12-j}\circ \Delta^G_{S^3} = {}^t (  \Delta^G_{S^3}\circ \pi^{S^3}_{j})\ \ \ \hbox{for\ }j=2,4\ .\]   
  In view of lemmas \ref{pi2} and \ref{pi4}, it follows that
   \begin{equation}\label{later}  \Delta^G_{S^3}\circ \pi^{S^3}_{12-j} -\gamma  \ \ \ \in A^6_{hom}(S^3\times S^3)\ \ \ \hbox{for\ }j=2,4\ ,\end{equation}
   where $\gamma$ is some cycle with support on $D\times S^3$ with $D\subset S^3$ a divisor. (Indeed, for $j=2$ one may take $\gamma={}^t (\gamma^{S^3}_{2})$, and for $j=4$
   one may take $\gamma={}^t (\gamma^{S^3}_4)$, which is supported on $(\hbox{codim.\ }2)\times(\hbox{codim.\ }4)$.)
   Applying the nilpotence theorem (theorem \ref{nilp}), it follows that there exists $N\in\NN$ such that
   \[   \Bigl( \Delta^G_{S^3}\circ \pi^{S^3}_{12-j} -\gamma \Bigr)^{\circ N}=0  \ \ \ \hbox{in}\ A^6_{}(S^3\times S^3)\ .\]
   Upon developing, this implies that
   \[       \bigl(\Delta^G_{S^3}\circ \pi^{S^3}_{12-j} \bigr)^{\circ N} = Q_1 + Q_2 + \cdots + Q_N   \ \ \ \hbox{in}\ A^6_{}(S^3\times S^3)\ ,\]
   where the $Q_i$ are compositions of correspondences in which $\gamma$ occurs at least once. The left--hand side is just $\Delta^G_{S^3}\circ \pi^{S^3}_{12-j}$ (since $    \Delta^G_{S^3}\circ \pi^{S^3}_{12-j} $ is idempotent, corollary \ref{idemp}). The right--hand side is supported on $D\times S^3$ (since $\gamma$ is), and so does not act on $0$--cycles. This proves equality (\ref{0cy}).
  
%  The case $A^6_{(6)}$ is proven similarly: using lemma \ref{pi6}, we find that
%    \[   \Delta^G_{S^3}\circ \pi_6^{S^3} -\Delta_{S^3} - \gamma_6^{S^3}\ \ \in A^{6}_{hom}(S^3\times S^3)\ ,\]
%    where $\gamma_6^{S^3}$ is some cycle with support on $D\times D$ for $D\subset S^3$ a divisor. Applying the nilpotence theorem, it follows there exists $N\in\NN$ such that
%    \[  \Bigl( \Delta^G_{S^3}\circ \pi_6^{S^3} -\Delta_{S^3} - \gamma_6^{S^3}\Bigr)^{\circ N}=0\ \ \ \hbox{in}\  A^{6}_{}(S^3\times S^3)\ .\]
%    Upon developing, this implies that
%    \begin{equation}\label{36}    \bigl( \Delta^G_{S^3}\circ \pi_6^{S^3} -\Delta_{S^3} \bigr)^{\circ N} = Q_1+Q_2+\cdots+Q_N\ \ \ \hbox{in}\  A^{6}_{}(S^3\times S^3)\ ,\end{equation}
%    where the $Q_i$ are correspondences composed with $\gamma_6^{S^3}$ and so supported on $D\times D$. It follows that the right--hand side does not act on $A^6(S^3)$. Using lemma \ref{comm}, we find that the left--hand side is $\pm  (\Delta^G_{S^3}\circ \pi_6^{S^3} -\Delta_{S^3})$, and so
%    \[      \bigl( \Delta^G_{S^3}\circ \pi_6^{S^3} -\Delta_{S^3} \bigr){}_\ast =0\colon\ \ \ A^6(S^3)\ \to\ A^6(S^3)  \ .\]
%    This proves
%    \[  (\Delta^G_{S^3})_\ast=\ide\colon\ \ \ A^6_{(6)}(S^3)\ \to\ A^6(S^3)\ ,\]
%    and thus (in view of the commutative diagram (\ref{diagsimp})) also
%    \[  (\Delta^G_{X})_\ast=\ide\colon\ \ \ A^6_{(6)}(X)\ \to\ A^6(X)\ .\]   
    
  We now consider the line $i=j$, i.e. the ``deepest part'' $A^i_{(i)}$ of the Chow groups. Diagram (\ref{comdiag3}) simplifies to
       \begin{equation}\label{diagsimp2} \begin{array}[c]{ccc}
       A^i_{(i)}(X) & \hookrightarrow & A^i_{(i)}(S^3)\\
        &&\\
          \ \ \ \  \downarrow {\scriptstyle (\Delta_X^G)_\ast} &&\ \ \ \downarrow {\scriptstyle (\Delta^G_{S^3})_\ast} \\
          &&\\
     A^i_{(i)}(X) & \hookrightarrow & A^i_{(i)}(S^3)\\
     \end{array}\end{equation}  
    
      In view of lemmas \ref{pi2} and \ref{pi4}, it follows that
   \[  \Delta^G_{S^3}\circ \pi^{S^3}_{i} -\gamma  \ \ \ \in A^6_{hom}(S^3\times S^3)\ ,\]
   where $\gamma$ is some cycle that acts trivially on $A^i(S^3)$. (Indeed, for $i=2$ one may take $\gamma=\gamma^{S^3}_{2}$, and for $i=4$
   one may take $\gamma=\gamma^{S^3}_4$.) Applying the nilpotence theorem, it follows there exists $N\in\NN$ such that
    \[  \Bigl( \Delta^G_{S^3}\circ \pi^{S^3}_{i} -\gamma \Bigr)^{\circ N}=0\ \ \ \hbox{in}\  A^{6}_{}(S^3\times S^3)\ .\]
    Upon developing, this implies that
    \begin{equation}\label{34}    \bigl( \Delta^G_{S^3}\circ \pi^{S^3}_{i}  \bigr)^{\circ N} = Q_1+Q_2+\cdots+Q_N\ \ \ \hbox{in}\  A^{6}_{}(S^3\times S^3)\ ,\end{equation}
    where the $Q_i$ are correspondences composed with $\gamma$. It follows that the right--hand side does not act on $A^6(S^3)$. The left--hand side is $  \Delta^G_{S^3}\circ \pi^{S^3}_{i}  $ (corollary \ref{idemp}), and so
    \[      ( \Delta^G_{S^3}){}_\ast =0\colon\ \ \ A^i_{(i)}(S^3)\ \to\ A^i(S^3) \ \ \ \hbox{for}\ i=2,4 \ .\]
   In view of the commutative diagram (\ref{diagsimp2})), it follows that also
    \[  (\Delta^G_{X})_\ast=0\colon\ \ \ A^i_{(i)}(X)\ \to\ A^i(X)\ \ \ \hbox{for}\ i=2,4\ .\]   
    
    We now consider $i=5$, i.e. $1$--cycles $A^{5}$. Diagram (\ref{comdiag3}) simplifies to
       \begin{equation}\label{diagsimp3} \begin{array}[c]{ccccc}
       A^5_{(j)}(X) & \hookrightarrow & A^5_{(j)}(S^3) &\oplus & A^4_{(j)}(S^2)   \\
        &&&&\\
          \ \ \ \  \downarrow {\scriptstyle (\Delta_X^G)_\ast} &&\ \ \ \downarrow {\scriptstyle (\Delta^G_{S^3})_\ast} &&  \ \ \ \downarrow {\scriptstyle (\Delta^G_{S^2})_\ast}    \\
          &&&&\\
     A^5_{(j)}(X) & \hookrightarrow & A^5_{(j)}(S^3) &\oplus&  A^4_{(j)}(S^2) \\
     \end{array}\end{equation}  
    
   For the $j=2$ case, we recall (equation (\ref{later})) that 
   \[     \Delta^G_{S^3}\circ \pi^{S^3}_{8} -\gamma  \ \ \ \in A^6_{hom}(S^3\times S^3)    \ ,\]
   where $\gamma$ is a cycle supported on $(\hbox{codim.\ }2)\times(\hbox{codim.\ }4)$. It follows that $\gamma$ does not act on $A^5$ (for dimension reasons). As before, applying the nilpotence theorem plus corollary \ref{idemp}, we find that
   \[     (  \Delta^G_{S^3}\circ \pi^{S^3}_{8})_\ast=0\colon\ \ \ A^5(S^3)\ \to\ A^5(S^3)    \ .\]
   This is equivalent to
   \begin{equation}\label{A5}   (  \Delta^G_{S^3})_\ast=0\colon\ \ \ A^5_{(2)}(S^3)\ \to\ A^5(S^3)  \ .\end{equation}
   
   Taking the transpose correspondences of lemma \ref{pi2} (and using lemma \ref{comm}), we also find
   \[ \Delta^G_{S^2}\circ \pi^{S^2}_6 - \gamma \ \ \ \in A^4_{hom}(S^2\times S^2)\ ,\]
   where $\gamma$ is a cycle supported on divisor times curve (indeed, one may take $\gamma={}^t \gamma^{S^2}_{2}$). Once more applying nilpotence (plus idempotence), we find that
   \[   (\Delta^G_{S^2}\circ \pi^{S^2}_6)_\ast =0\colon\ \ \ A^4_{}(S^2)\ \to\ A^4(S^2) \ ,\]
   which is equivalent to  
   \begin{equation}\label{A4}   (\Delta^G_{S^2})_\ast =0\colon\ \ \ A^4_{(2)}(S^2)\ \to\ A^4(S^2)\ .\end{equation}
   Combining equalities (\ref{A5}) and (\ref{A4}) implies that
   \[    (  \Delta^G_{X})_\ast=0\colon\ \ \ A^5_{(2)}(X)\ \to\ A^5(X)    \ ,\]
   in view of commutative diagram (\ref{diagsimp3}).
         
    Finally, the statement for $A^3_{(2)}$ follows from the commutative diagram
    \begin{equation}\label{final} \begin{array}[c]{ccccc}
       A^3_{(2)}(X) & \hookrightarrow & A^3_{(2)}(S^3) &\oplus & A^2_{(2)}(S^2)    \\
        &&&& \\
          \ \ \ \  \downarrow {\scriptstyle (\Delta_X^G)_\ast} &&\ \ \ \downarrow {\scriptstyle (\Delta^G_{S^3})_\ast} &&\ \ \ \downarrow {\scriptstyle (\Delta^G_{S^2})_\ast} \\
          &&&&\\
     A^3_{(2)}(X) & \hookrightarrow & A^3_{(2)}(S^3)  &\oplus & A^2_{(2)}(S^2)    \\   \\
     \end{array}  \end{equation}
  combined with the corresponding statement for $S^3$ and for $S^2$. The statement for $S^3$ is proven by recalling that (from the $i=4$ case of equality (\ref{34}) above)
  \[     \bigl( \Delta^G_{S^3}\circ \pi^{S^3}_{4}  \bigr)^{} = Q_1+Q_2+\cdots+Q_N\ \ \ \hbox{in}\  A^{6}_{}(S^3\times S^3)\ ,\]
  where the $Q_j$ are (composed with $\gamma^{S^3}_4$ and hence) supported on $(\hbox{codim.\ }4)\times(\hbox{codim.\ }2)$. For dimension reasons, the $Q_j$ act trivially on $A^3(S^3)$, and so
  \[     \bigl( \Delta^G_{S^3}\circ \pi^{S^3}_{4}  \bigr){}_\ast=0\colon\ \ \ A^3(S^3)\ \to\ A^3(S^3)\ .\]
  This is equivalent to
  \begin{equation}\label{32}   (\Delta^G_{S^3})_\ast=0\colon\ \ \ A^3_{(2)}(S^3)\ \to\ A^3(S^3)\ .\end{equation}
  The statement for $S^2$ is proven by noting that
  \[ \Delta^G_{S^2}\circ \pi_2^{S^2}- \gamma\ \ \ \in A^4_{hom}(S^2\times S^2)\ ,\]
  where $\gamma=\gamma^{S^2}_2$ is supported on divisor times divisor (lemma \ref{pi2}). Using nilpotence and idempotence, this implies
  \[  \Delta^G_{S^2}\circ \pi_2^{S^2}= Q_1+\cdots +Q_N\ \ \ \hbox{in}\ A^4(S^2\times S^2)\ ,\]
  where the $Q_j$ (are supported on divisor times divisor and hence) act trivially on $A^2_{(2)}(S^2)\subset A^2_{hom}(S^2)=A^2_{AJ}(S^2)$.  It follows that
  \[     \bigl( \Delta^G_{S^2}\circ \pi_2^{S^2}  \bigr){}_\ast=0\colon\ \ \ A^2_{(2)}(S^2)\ \to\ A^2(S^2)\ ,\]
  which is equivalent to
  \begin{equation}\label{22}   (\Delta^G_{S^2})_\ast=0\colon\ \ \ A^2_{(2)}(S^2)\ \to\ A^3(S^2)\ .\end{equation}
  
  Taken together, equations (\ref{32}) and (\ref{22}) imply that
  \[ (\Delta^G_X)_\ast=0\colon\ \ \ A^3_{(2)}(X)\ \to\ A^3(X)\ ,\]
  in view of diagram (\ref{final}).
    \end{proof}

\begin{remark} Let $X$ and $G$ be as in theorem \ref{main3}. Presumably, it is also possible to prove
  \[ A^6_{(6)}(X)\cap A^6(X)^G = A^6_{(6)}(X)\ ,\]
  in accordance with conjecture \ref{theconj}. Indeed, one can prove that
  \[ \Gamma:= (\Delta^G_{S^3}-\Delta_{S^3})\circ \pi_6^{S^3}  \ \ \ \in A^6(S^3\times S^3)  \]
  maps to $0$ under the restriction
  \[ H^{12}(S^3\times S^3)\ \to\ H^{12}\bigl( (S^3\times S^3)\setminus (V\times V)\bigr)\ ,\]
  where $V\subset S^3$ is some subvariety of codimension $2$.
  The problem is to find a cycle $\gamma$ supported on $V\times V$ and such that
  \[ \Gamma=\gamma\ \ \ \hbox{in}\ H^{12}(S^3\times S^3)\ ;\]
  that is, one needs to solve a special case of the ``Voisin standard conjecture'' \cite[Conjecture 1.6]{V0}.
  Perhaps, this can be done using the fact that $\rho(S)=20$ ? (I have tried a bit, then given up as things got messy...)
\end{remark}

\section{Some corollaries}

Theorem \ref{main3} can be extended to hyperk\"ahler varieties birational to $X$:

\begin{corollary} Let $X$ and $G$ be as in theorem \ref{main3}. Let $X^\prime$ be a hyperk\"ahler variety birational to $X$, and let $G^\prime$ denote the group of rational self--maps of $X^\prime$ induced by $G$. Then
  \[     A^i_{(j)}(X^\prime)\cap A^i(X^\prime)^{G^\prime} =
   %\begin{cases}
                                                                      %  A^i_{(j)}(X^\prime) &\hbox{if}\ (i,j)=(6,6)\ ;\\
                                                                            0 \ \ \hbox{if} \ (i,j)\in \  \Bigl\{ (2,2), (4,4), (3,2),(5,2),(6,2),(6,4)\Bigr\} .\]
                                                                 %   \end{cases}   \]
\end{corollary}

\begin{proof} This follows from theorem \ref{main3} combined with proposition \ref{birat}.
\end{proof}

\begin{corollary}\label{quot3} Let $X$ and $G\subset\aut(X)$ be as in theorem \ref{main3}. Let $Y$ be the quotient variety $Y:=X/G$. 

\noindent
(\rom1) $Y$ has a self--dual MCK decomposition.

\noindent
(\rom2) \[ \begin{split} A^i(Y)&= \bigoplus_{j\le 0} A^i_{(j)}(Y)\ \ \ \hbox{for}\ i\le 3\ ,\\
                                    A^5(Y)&=A^5_{(0)}(Y)\oplus A^5_{(4)}(Y)\ ,\\
                                    A^6(Y)&=A^6_{(0)}(Y)\oplus A^6_{(6)}(Y)\ .\\
                                    \end{split}\]
                                    
 %  \noindent
%   (\rom3) The quotient morphism $p\colon X\to Y$ induces isomorphisms
 %  \[ \begin{split}  p^\ast\colon\ \ A^6_{(6)}(Y)\ &\xrightarrow{\cong}\ A^6_{(6)}(X)\ ;\\
   %                       p^\ast\colon\ \ B^5_{(4)}(Y)\ &\xrightarrow{\cong}\ B^5_{(4)}(X)\ .\\
   %    \end{split}\]
     \end{corollary}
  
  \begin{proof} Point (\rom1) follows from lemma \ref{quotientmck}.                                          
      Point (\rom2) is just a translation of theorem \ref{main3}, combined with the fact that it is known that 
      \[ A^i_{(j)}(S^{[r]})=0\ \ \ \hbox{for\ }i\ge 2r-1\ \hbox{and\ }j<0\ .\]
      
  \end{proof}

Corollary \ref{quot3} has consequences for the multiplicative structure of the Chow ring of the quotient variety $Y$: 
                                    
\begin{corollary}\label{prod3}
Let $X$ and $G\subset\aut(X)$ be as in theorem \ref{main3}. Let $Y$ be the quotient variety $Y:=X/G$. 
For any $r\in\NN$, let 
  \[E^\ast(Y^r)\subset A^\ast(Y^r)\] 
  denote the subalgebra generated by (pullbacks of) $A^1(Y), A^2(Y), A^3(Y)$ and the diagonal $\Delta_Y\in A^6(Y\times Y)$ and the small diagonal $\Delta_Y^{sm}\in A^{12}(Y^3)$. Then the cycle class map
  \[ E^i(Y^r)\ \to\ H^{2i}(Y^r) \]
  is injective for $i\ge 6r-1$.
  
%  In particular,
%  \[ \begin{split} &\dim \ima\Bigl( A^i(Y)\otimes A^{6-i}(Y)\to A^6(Y)\Bigr)=1\ \ \ \forall 0<i<6\ ,\\
           %                & \ima\Bigl( A^3(Y)\otimes A^2(Y)\to A^5(Y)\Bigr)\cap A^5_{hom}(Y)=0\ .\\
               %         \end{split}\]
             \end{corollary}      
             
    \begin{proof} As we have seen (corollary \ref{quot3}(\rom1)), $Y$ has a self--dual MCK decomposition. Since the property of having a self--dual MCK decomposition is stable under products, $Y^r$ has a self--dual MCK decomposition, and so there is a bigraded ring structure $A^\ast_{(\ast)}(Y^r)$. We know (lemma \ref{diag}) that the diagonals $\Delta_Y$ and $\Delta_Y^{sm}$ are ``of pure grade $0$'', i.e.
    \[  \begin{split} \Delta_Y&\in\ \ A^6_{(0)}(Y\times Y)\ ,\\
                             \Delta_Y^{sm}&\in\ \ A^{12}_{(0)}(Y\times Y\times Y)\ .\\
                             \end{split}              \]
  We have also seen (corollary \ref{quot3}(\rom2)) that 
  \[ A^i(Y) = \bigoplus_{j\le 0} A^i_{(j)}(Y)\ \ \ \hbox{for}\ i\le 3\ .\]
  
    Consider now the projections $p_k\colon Y^r\to Y$ (on the $k$--th factor), and $p_{kl}\colon Y^r\to Y^2$ (on the $k$--th and $l$--th factor), and $p_{klm}\colon Y^r\to Y^3$ (on the $k$--th and $l$--th and $m$--th factor).      
    The projections $p_k, p_{kl}, p_{klm}$ 
    respect the bigrading of the Chow ring. (This follows from \cite[Corollary 1.6]{SV2}, or can be readily checked directly.) 
            
   It follows there is an inclusion
     \[ E^\ast(Y^r)\ \subset\  \bigoplus_{j\le 0} A^\ast_{(j)}(Y^r)\ ,\]
     and so in particular
     \[ E^i(Y^r)\ \subset\   A^i_{(0)}(Y^r)\ \ \ \hbox{for\ } i\ge 6r-1\ .\]
  As we have seen (lemma \ref{inj1}), the conjectural equality
    \begin{equation}\label{conjeq}  A^i_{(0)}(Y^r)\cap  A^i_{hom}(Y^r)\stackrel{??}{=}0 \end{equation}
    can be proven for $i\ge 6r-1$. 
    This proves the corollary.
       \end{proof}

%  The phenomenon displayed in corollary \ref{prod3} becomes even more pronounced when passing from rational to algebraic equivalence:    
                                  
%\begin{corollary}\label{prod3B}
%Let $X$ and $G\subset\aut(X)$ be as in theorem \ref{main3}. Let $Y$ be the quotient variety $Y:=X/G$. 
%For any $r\in\NN$, let 
%  \[F^\ast(Y^r)\subset B^\ast(Y^r)\] 
%  denote the subalgebra generated by (pullbacks of) $B^1(Y), B^2(Y), B^3(Y), B^4(Y)$ and the diagonal $\Delta_Y\in B^6(Y\times Y)$ and the small diagonal $\Delta_Y^{sm}\in B^{12}(Y^3)$. Then the cycle class map
%  \[ F^i(Y^r)\ \to\ H^{2i}(Y^r) \]
%  is injective for $i\ge 6r-1$.
  
%  In particular,
%  \[   \ima\Bigl( B^4(Y)\otimes B^1(Y)\to B^5(Y)\Bigr)\cap B^5_{hom}(Y)=0\ .\\
   %                   \]
     %        \end{corollary}              
                                  
%\begin{proof} This is the same argument as corollary \ref{prod3}; the only difference being that we now have
%  \[ B^i(Y) = B^i_{(0)}(Y)\ \ \hbox{for}\ i\le 4\ ,\] NOOOOOOOO !!!!!
%  which implies that $F^\ast(Y^r)$ is a subalgebra of $B^\ast_{(0)}(Y^r)$.
%\end{proof}

The phenomenon displayed in corollary \ref{prod3} becomes even more pronounced when restricting to the Chow ring of $Y$ (i.e., taking $r=1$):

\begin{corollary}\label{prod1} Let $X$ and $G$ be as in theorem \ref{main3}, and let $Y:=X/G$. Let $a\in A^i(Y)$ be a cycle with $i\not=3$. Assume $a$ is a sum of  intersections of $2$ cycles of strictly positive codimension, i.e.
  \[ a\in \ima\Bigl(  A^m(Y)\otimes A^{i-m}(Y)\ \to\ A^i(Y)\Bigr)\ ,\ \ \ 0<m<i\ .\]
  Then $a$ is rationally trivial if and only if $a$ is homologically trivial.
\end{corollary}

\begin{proof} Suppose $i=5$ or $i=6$. Since $A^r_{(r)}(Y)=0$ for $0<r<6$ (theorem \ref{main3}), we have
  \[ \begin{split}  \ima\Bigl(  A^m(Y)\otimes A^{i-m}(Y)\ \to\ A^i(Y)\Bigr) &=\ima \Bigl(  (\bigoplus_{j<m} A^m_{(j)}(Y))\otimes (\bigoplus_{j^\prime<i-m}A^{i-m}(Y))\ \to\ A^i(Y)\Bigr) \\
                                      &\subset \bigoplus_{j+j^\prime< i-1 } A^i_{(j+j^\prime)}(Y) =A^i_{(0)}(Y)\ .\\
                                    \end{split}\]
       The conclusion now follows from lemma \ref{inj1}.
       
   For $i=2$, the corollary follows from a far more general result of Voisin concerning intersections of divisors on Hilbert schemes of $K3$ surfaces \cite[Theorem 1.4]{V12}.
     
   It only remains to treat $i=4$. As both $m$ and $4-m$ are at most $3$, we have 
    \[   A^m(Y)=\bigoplus_{j\le 0} A^m_{(j)}(Y)\ ,\ \ \ A^{4-m}(Y)=\bigoplus_{j\le 0} A^{4-m}_{(j)}(Y)  \] 
    (theorem \ref{main3}). It follows that
     \[  \ima\Bigl(  A^m(Y)\otimes A^{4-m}(Y)\ \to\ A^4(Y)\Bigr)\ \subset\ \bigoplus_{j\le 0}  A^4_{(j)}(Y)\ ;\]
     the conclusion now follows from proposition \ref{inj2}.
                               \end{proof}

 \begin{proposition}\label{inj2} Let $X=(S_3)^{[3]}$  and $G\subset\aut(X)$ be as in theorem \ref{main3}. Let $Y:=X/G$. Then 
   \[ \begin{split}  &A^4_{(j)}(Y)=0\ \ \ \hbox{for\ }j<0\ ;\\
      &A^4_{(0)}(Y) \cap A^4_{hom}(Y)=0\ .\\
      \end{split}\]
    \end{proposition} 
    
    \begin{proof} First, observe that $A^4_{(j)}(Y)\to A^4_{(j)}(X)$ is split injective for any $j$ (this follows from the construction of the MCK decomposition for $Y$, lemma \ref{quotientmck}). Consequently, it suffices to prove that we have
    \[  \begin{split} &\bigl(A^4_{(j)}(X)\bigr)^G=0\ \ \ \hbox{for\ }j<0\ ;\\
         &\bigl(A^4_{(0)}(X)\bigr)^G\cap A^4_{hom}(X)=0\ .\\
         \end{split}\]    
   Let us first do the first statement. Using remark \ref{compat} plus the fact that
     \[ A^3_{(j)}(S^2)=A^2_{(j)}(S)=0\ \ \ \hbox{for\ }j<0\ ,\]
     we obtain for $j<0$ a commutative diagram
      \[ \begin{array}[c]{ccc}
                     A^4_{(j)}(X)& \hookrightarrow & A^4_{(j)}(S^3)\\
          &&\\
        \ \ \ \  \downarrow {\scriptstyle (\Delta_X^G)_\ast} &&\ \ \ \downarrow {\scriptstyle (\Delta^G_{S^3})_\ast} \\
         &&\\
         A^4_{(j)}(X)& \hookrightarrow & A^4_{(j)}(S^3)\ ,\\
         \end{array}    \]
    where horizontal arrows are split injections. We are thus reduced to proving that
    \[  (\Delta^G_{S^3})_\ast  A^4_{(j)}(S^3) = 0\ \ \ \hbox{for\ }j<0\ .\]
    Clearly $A^4_{(-4)}(S^3)=(\pi_{12}^{S^3})_\ast A^4(S^3)=0$. It is left to consider $j=-2$, i.e. we need to
     prove that
    \begin{equation}\label{injok0} (\Delta^G_{S^3}\circ \pi_{10}^{S^3})_\ast A^4(S^3) =0\ .\end{equation}    
     But we have seen that 
    \[ \Delta^G_{S^3}\circ \pi_{10}^{S^3} = {}^t (    \Delta^G_{S^3}\circ \pi_2^{S^3}  )\ \ \ \hbox{in}\ A^6(X\times X) \]
    (lemma \ref{idemp}), and so it follows from lemma \ref{pi2} that
    \[ \Delta^G_{S^3}\circ \pi_{10}^{S^3}-\gamma\ \in\ \ \ A^6_{hom}(X\times X)\ ,\]    
    where $\gamma$ is some cycle supported on $D\times C$, and $D$ is a divisor and $C\subset X$ is a curve.
    Applying the nilpotence theorem (plus the idempotence of lemma \ref{idemp}), we find
    \[ \Delta^G_{S^3}\circ \pi_{10}^{S^3}=Q_1+\cdots +Q_N\ \ \ \hbox{in}\  A^6_{}(X\times X)\ ,\]  
    where the $Q_j$ are supported on $D\times C$. For dimension reasons, the $Q_j$ act trivially on $A^4_{}(S^3)$ (indeed, the action of $Q_j$ on $A^4_{}(S^3)$ factors over $A^{-1}_{}(\wt{C})=0$).
    It follows that (\ref{injok0}) is true, proving the first statement of the proposition.    
    
   Next, let us prove the second part of the proposition. Since
    \[ A^3_{(0)}(S^2)\cap A^3_{hom}(S^2)=A^2_{(0)}(S)\cap A^2_{hom}(S)=0 \]
    (lemma \ref{inj1}), we obtain a commutative diagram
    \[ \begin{array}[c]{ccc}
                     A^4_{(0)}(X)\cap A^4_{hom}(X) & \hookrightarrow & A^4_{(0)}(S^3)\cap A^4_{hom}(S^3)\\
          &&\\
        \ \ \ \  \downarrow {\scriptstyle (\Delta_X^G)_\ast} &&\ \ \ \downarrow {\scriptstyle (\Delta^G_{S^3})_\ast} \\
         &&\\
         A^4_{(0)}(X)\cap A^4_{hom}(X) & \hookrightarrow & A^4_{(0)}(S^3)\cap \ \ A^4_{hom}(S^3)\ ,\\
         \end{array}    \]
    where horizontal arrows are split injections. We are thus reduced to proving that
    \[  (\Delta^G_{S^3})_\ast \Bigl( A^4_{(0)}(S^3)\cap A^4_{hom}(S^3)\Bigr) = 0\ ,\]
    which is equivalent to proving that
    \begin{equation}\label{injok} (\Delta^G_{S^3}\circ \pi_8^{S^3})_\ast A^4_{hom}(S^3) =0\ .\end{equation}
    But we have seen that 
    \[ \Delta^G_{S^3}\circ \pi_8^{S^3} = {}^t (    \Delta^G_{S^3}\circ \pi_4^{S^3}  )\ \ \ \hbox{in}\ A^6(X\times X) \]
    (lemma \ref{idemp}), and so it follows from lemma \ref{pi4} that
    \[ \Delta^G_{S^3}\circ \pi_8^{S^3}-\gamma\ \in\ \ \ A^6_{hom}(X\times X)\ ,\]    
    where $\gamma$ is some cycle supported on $W\times V\subset X\times X$, and $W\subset X$ is codimension $2$ and $V\subset X$ is codimension $4$. Applying the nilpotence theorem (plus the idempotence of lemma \ref{idemp}), we find
    \[ \Delta^G_{S^3}\circ \pi_8^{S^3}=Q_1+\cdots +Q_N\ \ \ \hbox{in}\  A^6_{}(X\times X)\ ,\]
    where the $Q_j$ are supported on $W\times V$. For dimension reasons, the $Q_j$ act trivially on $A^4_{hom}(S^3)$ (indeed, the action of $Q_j$ on $A^4_{hom}(S^3)$ factors over $A^0_{hom}(\wt{V})=0$).
    It follows that (\ref{injok}) is true, proving the second statement of the proposition.
    
 %   The $i=2$ case is similar. Since $A^1_{(0)}(S^2)\cap A^1_{hom}(S^2)=0$, we get from remark \ref{compat} a commutative diagram
   %   \[ \begin{array}[c]{ccc}
   %                  A^2_{(0)}(X)\cap A^2_{hom}(X) & \hookrightarrow & A^2_{(0)}(S^3)\cap A^2_{hom}(S^3)\\
   %       &&\\
    %    \ \ \ \  \downarrow {\scriptstyle (\Delta_X^G)_\ast} &&\ \ \ \downarrow {\scriptstyle (\Delta^G_{S^3})_\ast} \\
  %       &&\\
 %        A^2_{(0)}(X) \cap A^2_{hom}(X)& \hookrightarrow & A^2_{(0)}(S^3)\cap \ \ A^2_{hom}(S^3)\ ,\\
   %      \end{array}    \]
%    where horizontal arrows are split injections. We are thus reduced to proving that
%    \begin{equation}\label{injok2}  (\Delta^G_{S^3}\circ \pi_4^{S^3})_\ast A^2_{hom}(S^3) =0\ .   \end{equation}    
%    From equation (\ref{thisone}), combined with lemma \ref{idemp}, one obtains an equality
%    \[  \Delta^G_{S^3}\circ \pi_4^{S^3}={}^t Q_1+\cdots +{}^t Q_N\ \ \ \hbox{in}\  A^6_{}(X\times X)\ , \]
%    where the ${}^t Q_j$ are supported on $V\times W$. For dimension reasons, the ${}^t Q_j$ act trivially on $A^2_{hom}(S^3)$ (indeed, the action of ${}^t Q_j$ on %$A^2_{hom}(S^3)$ factors over $A^0_{hom}(\wt{W})=0$), and so
 %   \[ (\Delta^G_{S^3}\circ \pi_4^{S^3})_\ast A^2_{hom}(S^3) =0\ .   \]
 %   This proves equality (\ref{injok2}), and hence the proposition for $i=2$.
          \end{proof}

\begin{remark}\label{compare} Corollaries \ref{prod3} and \ref{prod1} are similar to the Beauville--Voisin conjecture, on the one hand, and to results of Voisin and L. Fu for Calabi--Yau varieties, on the other hand. 

The Beauville--Voisin conjecture \cite[Conjecture 1.3]{V12} concerns the Chow ring of a hyperk\"ahler variety $X$. The conjecture is that the subring
  \[ D^\ast(X)\ \subset\ A^\ast(X) \]
  generated by divisors and Chern classes injects (via the cycle class map) into cohomology. Partial results towards this conjecture have been obtained in \cite{V12}, 
  \cite{Rie2}, \cite{Yin}.
  
  On the other hand, if $Y$ is a Calabi--Yau variety that is a generic complete intersection, say of dimension $n$, it has been proven that the image of the intersection product
  \[ \ima\Bigl(  A^i(Y)\otimes A^{n-i}(Y)\ \to\ A^n(Y) \Bigr)\ ,  \ \ 0<i<n\ ,\]
  is of dimension $1$ and hence injects into cohomology \cite{V13}, \cite{LFu}.
  
  Results like corollaries \ref{prod3} and \ref{prod1} are presumably {\em not\/} true for all Calabi--Yau varieties (since not all Calabi--Yau varieties verify Beauville's weak splitting property \cite{Beau3}); for a general Calabi--Yau variety, one only expects statements about $0$--cycles. Conjecturally, statements concerning other codimensions (such as corollaries \ref{prod3} and \ref{prod1}) should be true for Calabi--Yau varieties that are finite quotients of hyperk\"ahler varieties.
\end{remark}

\section{A partial generalization}

This section contains a partial generalization of theorem \ref{main3}. We consider Hilbert schemes $X=(S_k)^{[k]}$, where $S_k$ is any of the LSY surfaces. The same result (theorem \ref{main}) also applies to some of the Sch\"utt surfaces.

\begin{theorem}\label{main} Let $S_k$ be an LSY surface, or a Sch\"utt surface as in proposition \ref{schuttfdim}, with $G_{S_k}\subset
\aut(S_k)$ the order $k$ group of non--symplectic automorphisms
of definition \ref{lsy}, resp. definition \ref{schu}. 
 Let $X$ be the Hilbert scheme $X:=(S_k)^{[k]}$ of dimension $n=2k$. Let $G\subset \aut(X)$ be the order $k$ group of non--symplectic natural automorphisms, corresponding to $G_{S_k}\subset
\aut(S_k)$. Then
\[      A^i_{(2)}(X)\cap A^i(X)^G=  0\ \ \ \hbox{for}\ i\in\{2,n\}\ .  \]
\end{theorem}

\begin{proof} Let us write $S$ for the surface $S_k$. Let $i\in\{2,n\}$. Using remark \ref{compat}, one finds a commutative diagram
  \[ \begin{array}[c]{ccc}
       A^i_{(2)}(X) &\to&  A^i_{(2)}(S^k)\\
       &&\\
      \ \ \ \ \  \downarrow {(\Delta^G_X)_\ast} &&  \ \ \ \ \  \downarrow {(\Delta^G_{S^k})_\ast}\\
      &&\\
       A^i_{(2)}(X) &\to&  \ \ A^i_{(2)}(S^k)\ ,\\
       \end{array}\]
   where horizontal arrows are split injections. Here $\Delta^G_{S^k}$ is as before defined as the projector
    \[  \Delta^G_{S^k}:= {1\over k}{\displaystyle\sum_{g\in G_{S}}} \Gamma_g\times\cdots\times\Gamma_g\in A^{2k}(S^k\times S^k)\ .\]
   
   We are thus reduced to proving that
   \begin{equation}\label{need} (\Delta^G_{S^k})_\ast=0\colon\ \ \  A^i_{(2)}(S^k)\ \to\ A^i_{(2)}(S^k)\ \ \ \hbox{for}\ i\in\{2,n\}\ .\end{equation}   
       
   Let us assume $i=2$. Lemma \ref{pi2k} (with $k=r$) implies that
   \[   \Delta^G_{S^k}\circ \pi_2^{S^k}  -\gamma\ \in\ \ \ A^{2k}_{hom}(S^k\times S^k)\ ,\]
   where $\gamma$ is a cycle supported on (curve)$\times$(divisor).
   But $S^k$ has finite--dimensional motive, and so there exists $N\in\NN$ such that
   \[   \Bigl(\Delta^G_{S^k}\circ \pi_2^{S^k}  -\gamma\Bigr)^{\circ N}=0\ \ \ \hbox{in}\  A^{2k}_{}(S^k\times S^k)\ .\]   
   Developing, and using that $\Delta^G_{S^k}\circ \pi_2^{S^k}$ is idempotent (lemma \ref{idemp}), this implies that
   \begin{equation}\label{trans}    \Delta^G_{S^k}\circ \pi_2^{S^k} = Q_1 +\cdots +Q_N\ \ \ \hbox{in}\  A^{2k}_{}(S^k\times S^k)\ ,\end{equation}
   where each $Q_j$ is supported on $C\times D$ and hence does not act on $A^2_{hom}(S^k)=A^2_{AJ}(S^k)$. It follows that
   \[  (\Delta^G_{S^k}\circ \pi_2^{S^k})_\ast =0\colon\ \ \ A^2_{hom}(S^k)\ \to\ A^2(S^k)\ ,\]
   and thus
   \[      (\Delta^G_{S^k})_\ast =0\colon\ \ \ A^2_{(2)}(S^k)\ \to\ A^2(S^k)\ ,\]
   proving (\ref{need}) for $i=2$.
   
   It remains to consider the case $i=n$. Taking the transpose of equality (\ref{trans}) and invoking lemma \ref{idemp}, one obtains an equality
   \[  \Delta^G_{S^k}\circ \pi_{4k-2}^{S^k} =   \pi^{S^k}_{4k-2}\circ \Delta^G_{S^k}=    {}^t Q_1 +\cdots +{}^t Q_N\ \ \ \hbox{in}\  A^{2k}_{}(S^k\times S^k)\ ,\]
   where the ${}^t Q_j$ are supported on $D\times C$. The ${}^t Q_j$ do not act on $A^n(S^k)$ (for dimension reasons), and so
   \[ ( \Delta^G_{S^k}\circ \pi_{4k-2}^{S^k}  )_\ast=0\colon\ \ \ A^n(S^k)\ \to\ A^n(S^k)\ ,\]
   proving (\ref{need}) for $i=n$. 
   \end{proof}
   
 Theorem \ref{main} has implications for the quotient $Y:=X/G$ (the variety $Y$ is a ``Calabi--Yau variety with quotient singularities''):  
   
\begin{corollary}\label{cor1} Let $X$ and $G$ be as in theorem \ref{main}, and let $Y:=X/G$ be the quotient. For any $r\in\NN$, let
  \[ E^\ast(Y^r)\ \subset\ A^\ast(Y^r) \]
  be the subalgebra generated by (pullbacks of) $A^1(Y)$ and $A^2(Y)$ and $\Delta_Y$, $\Delta_Y^{sm}$. Then the cycle class map induces maps
  \[ E^i(Y^r)\ \to\ H^{2i}(Y^r) \]
  that are injective for $i\ge nr-1$.
  \end{corollary}
  
  \begin{proof} This is similar to corollary \ref{prod3}. First, it follows from lemma \ref{quotientmck} that $Y$, and hence $Y^r$, has a self--dual MCK decomposition. Consequently, the Chow ring $A^\ast(Y^r)$ is a bigraded ring. Theorem \ref{main} (plus the fact that $A^1_{hom}(Y)=0$) implies that
  \[  A^i(Y)=\bigoplus_{j\le 0} A^i_{(j)}(Y)\ \ \ \hbox{for}\ i\le 2\ .\]
  Lemma \ref{diag} ensures that
  \[ \Delta_Y\in A^n_{(0)}(Y)\ ,\ \ \ \Delta_Y^{sm}\in A^{2n}(Y^3)\ .\]
  Since pullbacks for projections of type $Y^r\to Y^{s}, s<r$, preserve the bigrading (this follows from \cite[Corollary 1.6]{SV2}, or can be readily checked directly), this implies that
  \[ E^\ast(Y^r)\subset \bigoplus_{j\le 0}A^\ast_{(j)}(Y^r)\ .\]
  In particular, this implies
  \[ E^i(Y^r)\subset A^i_{(0)}(Y^r)\ \ \ \hbox{for\ }i\ge nr-1\ .\]
 The corollary now follows from the fact that
  \[  A^i_{(0)}(Y^r)\cap A^i_{hom}(Y^r)\ \to\  A^i_{(0)}(X^r)\cap A^i_{hom}(X^r)  \]
  is injective (this is true for any $i$), and 
  \[ A^i_{(0)}(X^r)\cap A^i_{hom}(X^r) =0\ \ \ \hbox{for}\ i\ge nr-1 \ \]
  (lemma \ref{inj1}).
    \end{proof}
  
 \begin{corollary}\label{cor2} Let $X$ and $G$ be as in theorem \ref{main}, and let $Y:=X/G$ be the quotient. Let $a\in A^n(Y)$ be a $0$--cycle which is in the image of the intersection product map
   \[  A^3(Y)\otimes A^{i_1}(Y)\otimes\cdots\otimes A^{i_s}(Y)\ \to\ A^n(Y)\ ,\]
   with all $i_m\le 2$ (and $i_1+\cdots+i_s=n-3$). Then $a$ is rationally trivial if and only if $\deg(a)=0$.
   \end{corollary}

\begin{proof} The point is that
   \[  \begin{split} A^3(Y)&=\bigoplus_{j\le 0} A^3_{(j)}(Y)\oplus A^3_{(2)}(Y)\ ,\\
                           A^{i_m}(Y)&=\bigoplus_{j\le 0} A^{i_m}_{(j)}(Y)\ \ \ \hbox{for}\ i_m\le 2\ \\
            \end{split}\]
         (theorem \ref{main}),   
     and so
     \[ a\in A^n_{(0)}(Y)\oplus A^n_{(2)}(Y)\ .\]
     But we have seen that $A^n_{(2)}(Y)=0$ (theorem \ref{main}), and so 
     \[         a\in A^n_{(0)}(Y)\cong\QQ\ .\]
    \end{proof}

\vskip1cm
\begin{nonumberingt} Thanks to all participants of the Strasbourg 2014/2015 ``groupe de travail'' based on the monograph \cite{Vo} for a very stimulating atmosphere.
Many thanks to Yasuyo, Kai and Len for smoothly running the Schiltigheim Math Research Institute.
\end{nonumberingt}

\end{document}